\documentclass[reqno, 12pt]{amsart}

\usepackage{amsmath,graphicx}

\usepackage[algoruled,lined,noresetcount,norelsize]{algorithm2e}

\usepackage{amssymb,amsthm}
\usepackage{mathrsfs}
\usepackage{mathabx}\changenotsign
\usepackage{dsfont}
\usepackage{xcolor}
\usepackage[backref]{hyperref}
\hypersetup{
	colorlinks,
	linkcolor={red!60!black},
	citecolor={green!60!black},
	urlcolor={blue!60!black}
}
\usepackage[T1]{fontenc}
\usepackage{lmodern}
\usepackage[babel]{microtype}
\usepackage[english]{babel}

\usepackage{geometry}
\numberwithin{equation}{section}
\numberwithin{figure}{section}

\usepackage{enumitem}

\theoremstyle{plain}
\newtheorem{theorem}{Theorem}[section]
\newtheorem{corollary}[theorem]{Corollary}
\newtheorem{lemma}[theorem]{Lemma}
\newtheorem{conjecture}[theorem]{Conjecture}

\newtheorem{problem}[theorem]{Problem}
\newtheorem{claim}[theorem]{Claim}

\renewcommand{\epsilon}{\varepsilon}

\newcommand{\Exp}{\mathbb{E}}
\newcommand{\Prob}{\mathbb{P}}

\begin{document}

	\title[Tree universality in positional games]{%
		Tree universality in positional games}
	
	\author[Grzegorz Adamski]{Grzegorz Adamski}
	\author[Sylwia Antoniuk]{Sylwia Antoniuk}
	\author[Ma\l gorzata Bednarska-Bzd\c{e}ga]{%
		Ma\l gorzata Bednarska-Bzd\c{e}ga}
	\author[Dennis Clemens]{Dennis Clemens}
	\author[Fabian Hamann]{Fabian Hamann}
	\author[Yannick Mogge]{Yannick Mogge}
	
	\thanks{The research of the fourth and sixth author is supported 
		by Deutsche Forschungsgemeinschaft (Project CL 903/1-1).}
	
	\address{Department of Discrete Mathematics, Faculty of Mathematics 
		and CS, Adam Mickiewicz University, Pozna\'n, Poland}
	\email{grzegorz.adamski@amu.edu.pl, sylwia.antoniuk@amu.edu.pl, 
		mbed@amu.edu.pl}
	\address{Institute of Mathematics, Hamburg University of Technology, 
		Hamburg, Germany.}
	\email{dennis.clemens@tuhh.de, fabian.hamann@tuhh.de,
		yannick.mogge@tuhh.de}

	\pagestyle{plain}
	
	\maketitle
	
	\begin{abstract}
		In this paper we consider positional games
		where the winning sets are tree universal graphs.
		Specifically, we show that in the unbiased Maker-Breaker
		game on the complete graph $K_n$,
		Maker has a strategy to occupy a graph which contains
		copies of all spanning trees with maximum degree at most $cn/\log(n)$, for a suitable constant $c$ and $n$ being large enough. We also prove an analogous result for
		Waiter-Client games. Both of our results show that the building player can play at least as good as
		suggested by the random graph intuition. Moreover,
		they improve on a special case of earlier results
		by Johannsen, Krivelevich, and Samotij 
		as well as Han and Yang for Maker-Breaker games.
	\end{abstract}

	\section{Introduction}
	A positional game is a perfect-information game played by two players on a hypergraph denoted as $\mathcal{H} = (\mathcal{X}, \mathcal{F})$, where $\mathcal{X}$ is called the board, and $\mathcal{F}$ is a family of winning sets. In this type of game, both players claim elements of the board $\mathcal{X}$, following predefined rules. The victor is determined based on the family of winning sets $\mathcal{F}$. Over the past decade, positional games were extensively studied (for a comprehensive overview, refer to \cite{hefetz2014positional}), and various variants have been considered. In this paper, we focus on unbiased Maker-Breaker and Waiter-Client games played on the edge set of the complete graph $K_n$ with the winning sets being \textit{tree-universal} graphs on $n$ vertices, i.e. graphs which contain a copy of every tree $T$ on $n$ vertices with the maximum degree $\Delta(T) \leq \Delta(n)$ bounded by a suitable function on $n$. Our results improve a result by Johannsen, Krivelevich, and Samotij~\cite{johannsen2013expanders} from 2013, and make progress in answering a question by Ferber, Hefetz and Krivelevich~\cite{ferber2012fast} from 2012.
	
	\subsection{Maker-Breaker games concerning spanning trees}
	A $(1:b)$ Maker-Breaker game on some hypergraph $\mathcal{H} = (\mathcal{X},\mathcal{F})$ is played as follows: Maker and Breaker take turns claiming elements of the board $\mathcal{X}$. Maker always takes one element per turn while Breaker takes $b$ elements, except perhaps in his final move. The value of $b$ is referred to as the bias. Maker wins if she successfully claims an entire winning set $F \in \mathcal{F}$, otherwise Breaker wins. It is observed easily that it is only beneficial for Breaker to claim more elements, and thus there is a threshold bias $b_{\mathcal{H}}$ such that Breaker wins if and only if $b > b_{\mathcal{H}}$ (excluding degenerate cases).
	
	Already in 1964, Lehman~\cite{lehman1964solution} discovered that Maker easily wins the $(1:1)$ connectivity game $\mathcal{C}$ on $K_n$, i.e.~the game where the winning sets consist of all spanning trees of $K_n$, and she can even do so, if the board only consists of two edge-disjoint spanning trees. There are different natural related questions which have been investigated since then. Chv\'atal and Erd\H{o}s~\cite{chvatal1978biased} proved in 1978 that the threshold bias $b_{\mathcal{C}}$ for the connectivity game on $K_n$ is of order $\frac{n}{\log(n)}$, and suspected that there is an interesting relation between Maker-Breaker games and random graphs. They conjectured that for certain biased Maker-Breaker games the more likely winner between two random players would be the same as the winner between two perfect players. This relation is commonly referred to as the \textit{random graph intuition}, and holds for several different Maker-Breaker games. Indeed, Gebauer and Szab\'o~\cite{gebauer2009asymptotic} showed in 2009 that $b_{\mathcal{C}} = (1 + o(1))\frac{n}{\log(n)}$ which confirms the random graph intuition for the connectivity game.
	
	Going one step further, one can also inspect whether Maker can claim a copy of a fixed spanning tree instead of just any spanning tree and if so, how fast she can do it. In particular, an intriguing question was asked by Ferber, Hefetz, and Krivelevich~\cite{ferber2012fast} in 2012: What is the largest $d = d(n) \in \mathbb{N}$ such that Maker can claim any fixed tree $T$ on $n$ vertices with maximum degree $\Delta(T) \leq d$ for $n$ large enough? 
	
	In the unbiased case, following~\cite{krivelevich2010embedding}, the random graph intuition would suggest $d(n) = \Theta(\frac{n}{\log(n)})$. However, the current best-known results are quite far away from this desired value. 
	In 2009, Hefetz, Krivelevich, Stojakovi\'c, and Szab\'o~\cite{hefetz2009fast} showed that Maker can claim a Hamilton path in the $(1:1)$ game within $n-1$ rounds. Similarly, Maker can claim any fixed tree $T$ with $\Delta(T) = O(1)$ within $n+1$ rounds, provided $n$ is large enough~\cite{clemens2015building}. 
	In 2012, Ferber, Hefetz and Krivelevich~\cite{ferber2012fast} showed that Maker can claim any fixed tree $T$ with $\Delta(T) \leq n^{0.05}$ within $n + o(n)$ rounds, even in the biased version as long as $b < n^{0.005}$. Johannsen, Krivelevich, and Samotij~\cite{johannsen2013expanders} further improved this maximum degree for the unbiased setting, where their result (which is stated more generally for expander graphs) is \textit{universal}, giving that in a $(1:1)$ game on $K_n$, Maker can claim a single graph containing copies of all trees $T$ on $n$ vertices such that $\Delta(T) \leq \frac{c n^{1/3}}{\log(n)}$, for some suitable $c$ and large enough $n$.
	
	We further improve on this result, and show that Maker can play asymptotically at least as good as the random graph intuition suggests.
	
	\begin{theorem}\label{thm:MB.tree.universal}
		There exists a constant $c>0$ such that the following holds for every large enough integer $n$.
		In the $(1:1)$ Maker-Breaker game on $K_n$,
		Maker has a strategy to occupy a graph which contains 
		a copy of every tree $T$ with $n$ vertices
		and maximum degree $\Delta(T)\leq \frac{cn}{\log(n)}$.
	\end{theorem}
	
	Our proof technique is different from the one in~\cite{johannsen2013expanders}. In~\cite{johannsen2013expanders}, Maker in a $(1:1)$ game on $K_n$ builds a proper expander based on the Erd\H os-Selfridge criterion for Breaker’s win (see Lemma~\ref{lem:ES-criterion} in the next section). A natural way for obtaining a stronger result with this method would be to show stronger universality properties of expanders. Han and Yang~\cite{han2023expanders} went this route and showed that Maker in  a $(1:1)$ game on $K_n$ can build a graph containing copies of all spanning trees $T$ with $\Delta(T) \leq \frac{c n^{1/2}}{\log(n)}$. In our proof of Theorem~\ref{thm:MB.tree.universal}, Maker builds a graph having not only good expanding properties, but also other properties, which cannot be obtained by the Erd\H os-Selfridge criterion. Let us add that an advantage of the method in~\cite{johannsen2013expanders} is that it generalizes easily to biased Maker-Breaker games (in fact the authors present their result for biased games, played on expanders), while our method is less flexible in that sense. 
	
	\subsection{Waiter-Client games concerning spanning trees}
	A $(1:b)$ Waiter-Client game on some hypergraph $\mathcal{H} = (\mathcal{X},\mathcal{F})$ is played as follows: in each round, Waiter picks $b + 1$ elements of the board $\mathcal{X}$ and offers them to Client. Client chooses one of them for himself and returns the rest to Waiter. 
	If in the last round there are less than $b+1$ elements not selected yet, all the elements go to Waiter. Waiter wins if she is able to force Client to fully claim some winning set $F \in \mathcal{F}$. Otherwise, Client wins.
	
	Concerning spanning trees, similar results are known for Waiter-Client games as for Maker-Breaker games, where Waiter often can achieve better results because she has more control which edges get blocked.
	Similar to Maker-Breaker games, Waiter wins the $(1:1)$ connectivity game on some graph $G$ if $G$ contains two edge-disjoint spanning trees~\cite{csernenszky2009chooser}. Concerning the threshold bias, the third author together with Krivelevich and \L uczak~\cite{bednarska2016manipulative} showed that Waiter wins the $(1:b)$ Waiter-Client connectivity game on $K_n$ if and only if $b \leq \lfloor \frac{n}{2} \rfloor - 1$. 
	Moreover, in a $(1:1)$ Waiter-Client game on $K_n$ with $n$ large enough, Waiter can force Client to claim a Hamilton path within $n-1$ rounds, and she can force Client to claim any fixed tree $T$ with $\Delta(T) \leq c \sqrt{n}$ within $n$ rounds,
	for some suitable $c$ and $n$ large enough~\cite{clemens2020fast}. 
	Recently, we improved upon this result showing that Waiter can force any tree $T$ with $\Delta(T) < (\frac{1}{3} - o(1))n$ \cite{fixedtree}.
	
	Similar to the Maker-Breaker case, it is known that Waiter in the $(1:1)$ game on $K_n$ can force Client to claim a graph that contains copies of all trees $T$ with $\Delta(T) \leq \frac{c n^{1/3}}{\log(n)}$, 
	or even $\Delta(T) \leq \frac{c n^{1/2}}{\log(n)}$, which follows from the above mentioned universality properties of expanders proved in~\cite{johannsen2013expanders}, \cite{han2023expanders} and the fact that the Erd\H os-Selfridge criterion has its Waiter-Client counterpart (see~\cite{bednarska2013weight}). We again improve upon this result by showing a Waiter-Client version of Theorem~\ref{thm:MB.tree.universal}.
	
	\begin{theorem}\label{thm:WC.tree.universal}
		There exists a constant $c>0$ such that the following holds for every large enough integer $n$.
		In the $(1:1)$ Waiter-Client game on $K_n$,
		Waiter has a strategy to force Client to claim a graph which contains 
		a copy of every tree $T$ with $n$ vertices
		and maximum degree $\Delta(T)\leq \frac{cn}{\log(n)}$.
	\end{theorem}
	
	\medskip
	
	\textbf{Organisation of the paper.}
	In Section~\ref{sec:preliminaries}
	we collect useful tools from probability, for positional games and for embedding trees.
	In Section~\ref{sec:tree.universal}
	we prove a sufficient condition
	for a graph to be universal for trees of large maximum degree. In Section~\ref{sec:maker.strategy}
	and in Section~\ref{sec:waiter.strategy}
	we then show that Maker and Waiter, respectively,
	have a strategy for creating such a graph,
	hence proving Theorem~\ref{thm:MB.tree.universal}
	and Theorem~\ref{thm:WC.tree.universal}.
	We add some concluding remarks in Section~\ref{sec:concluding}, in which we consider also the tree universality problem in Client-Waiter and Avoider-Enforcer games. 
	
	\medskip
	
	\subsection{Notation} 
	Most of our notation is standard 
	and follows that of~\cite{west2001introduction}.
	First of all, we set $[n]:=\{k\in\mathbb{N}:~ 1\leq k\leq n\}$ for
	every positive integer $n$. 
	Let $G$ be any graph. Then we write $V(G)$ and $E(G)$ for the vertex  set and the edge set of $G$, respectively, and we set 
	$v(G) :=|V(G)|$ and $e(G):=|E(G)|$.
	If $\{v,w\}$ is an edge in $G$,  
	we shortly write $vw$. 
	The neighborhood of $v$ is
	$N_G(v) : =\{w\in V(G): vw\in E(G)\}$,
	and its degree is
	$d_G(v) :=|N_G(v)|$.
	The maximum degree and minimum degree in $G$ are denoted
	$\Delta(G)$ and $\delta(G)$, respectively.
	Given any $A,B\subset V(G)$ and $v\in V(G)$,
	we write 
	$d_G(v,A):=|N_G(v)\cap A|$,
	$N_G(A):= \left(\bigcup_{v\in A} N_G(v)\right)\setminus A$, 
	$E_G(A):=\{vw\in E(G):~ v,w\in A\}$,
	$e_G(A):=|E_G(A)|$,
	$E_G(A,B):=\{vw\in E(G):~ v\in A,w\in B\}$,
	and
	$e_G(A,B):=|E_G(A,B)|$.
	Note that whenever the graph $G$ is clear from the context,
	we may omit the subscript $G$ in all definitions above.
	Given any $A\subset V(G)$, we let $G[A]=(A,E_G(A))$  
	be the subgraph of $G$ induced by $A$,
	and we set $G-A:=G[V(G)\setminus A]$.

	Let $G$ and $H$ be any graphs. 
	We write $H\subset G$ if both $V(H)\subset V(G)$ and 
	$E(H)\subset E(G)$ hold, and then call $H$ a subgraph of $G$.
	We say that $H$ and $G$ are isomorphic
	if there exists a bijection $b:V(H)\rightarrow V(G)$
	such that $vw$ is an edge of $H$ if and only if
	$b(v)b(w)$ is an edge of $G$. In this case we  
	also say that $H$ forms a copy of $G$. 
	An embedding of $H$ into $G$ is an injective map 
	$f:V(H)\rightarrow V(G)$ such that $vw\in E(H)$ implies
	$f(v)f(w)\in E(G)$.
	
	Let $T$ be a tree. Then we write $L(T)$ for the set of leaves, i.e.~all vertices of degree 1 in $T$. Moreover, a path $P$ in $T$ is called a bare path if all of its inner vertices
	have degree 2 in $T$. 
	
	Assume that some Waiter-Client game is in progress, then
	we let $C$ denote the graph consisting of 
	Client's edges only, with the vertex set being equal 
	to the graph that the game is played on.
	Similarly, in a Maker-Breaker game, we use $M$
	and $B$
	to denote the graph consisting of Maker's or Breaker's edges only.
	If an edge belongs to any player in the game, 
	then we call it claimed. Otherwise, we say that the edge is free. The graph consisting of the free edges will always be denoted $F$.
	
	We write Bin$(n,p)$ for the binomial random variable
	with $n$ trials, each having success independently 
	with probability $p$. We write $X\sim \text{Bin}(n,p)$ to   
	denote that $X$ is distributed like Bin$(n,p)$.
	We say that an event, depending on $n$,
	holds asymptotically almost surely (a.a.s.)
	if it holds with probability tending to 1 if $n$
	tends to infinity. For functions $f,g:\mathbb{N} \rightarrow \mathbb{R}$, we write $f(n)=o(g(n))$ if
	$\lim_{n\rightarrow \infty} |f(n)/g(n)| = 0$.
	All logarithms are with respect to basis $e$.

	\bigskip

	\section{Preliminaries}\label{sec:preliminaries}

	\subsection{Probabilistic tools}
	In some of our probabilistic arguments, we will use
	Chernoff bounds (see e.g.~\cite{janson2011random}) 
	to show concentration for binomially distributed random variables. 
	Specifically, we will use the following.
	
	\begin{lemma}\label{lem:Chernoff}
		If $X \sim \text{Bin}(n,p)$, then
		\begin{itemize}
			\item $\Prob(X<(1-\delta)np)
			< \exp\left(-\frac{\delta^2np}{2}\right)$ 
			for every $\delta>0$, and
			\item $\Prob(X>(1+\delta)np)< \exp\left(-\frac{np}{3}\right)$ 
			for every $\delta\geq 1$.
		\end{itemize}
	\end{lemma}
	
	\medskip

	\subsection{Maker-Breaker game tools}
	
	For the discussion of the Maker-Breaker tree universality game,
	we will use several tools from positional games theory.
	The first one is the famous Erd\H{o}s-Selfridge-Criterion~\cite{erdos1973combinatorial},
	stated as e.g.~in Theorem 2.3.3 in~\cite{hefetz2014positional}.
	
	\begin{lemma}[Erd\H{o}s-Selfridge-Criterion~\cite{erdos1973combinatorial}]\label{lem:ES-criterion}
		Let $(X,\mathcal{F})$ be a hypergraph satisfying
		$$\sum_{F\in \mathcal{F}} 2^{-|F| + 1} < 1.$$
		Then in the $(1:1)$ Maker-Breaker game on 
		$(X,\mathcal{F})$, Breaker has a strategy to
		claim at least one element in each of the
		winning sets in $\mathcal{F}$. 
	\end{lemma}
	
	We will also use statements ensuring that Maker can achieve sufficiently large degrees.

	\begin{lemma}[Mindegree game, Lemma 10 in~\cite{hefetz2009sharp}]\label{lem:MinDeg}
		Let $H$ be a graph of minimum degree $d$, then in a 
		$(1:1)$ Maker-Breaker game played on the edges of $H$, 
		Maker can build a spanning graph $M$ with minimum degree
		at least $\lfloor d/4 \rfloor$.
	\end{lemma}

	\begin{lemma}[Degree game, Corollary of Lemma 6 in~\cite{balogh2009diameter}]\label{lem:Deg}
		Playing a $(1:2)$ Maker-Breaker game on the edges of $K_n$, Maker can ensure that every vertex reaches
		degree at least
		$\frac{1}{3}n - 3\sqrt{n\log(n)}$
		in her graph.
	\end{lemma}

	\medskip
	
	Moreover, we will use the following lemma which allows Maker
	to distribute her elements nicely over all sets of
	a given family $\mathcal{F}$.
	
	\begin{lemma}[Corollary of Lemma 2.3 in~\cite{barkey2023multistage}]\label{lem:criterion.multistage}
		Let $X$ be a set and let $\delta \in (0,1)$. Let $\mathcal{H} = (X, \mathcal{F})$ be a hypergraph, and $k = \min_{F \in \mathcal{F}} |F|$. 
		If $k > 4\delta^{-2}\ln(|\mathcal{F}|)$, 
		then in a $(1:1)$ Maker-Breaker game on $\mathcal{H}$, 
		Maker has a strategy to claim at least 
		$(\frac{1}{2} - \delta)|F|$ elements of every set 
		$F \in \mathcal{F}$.
	\end{lemma}
	
	Finally, we will use the following corollary of
	a recent result by 
	Liebenau and Nenadov~\cite{liebenau2022threshold}.
	
	\begin{lemma}[$K_5$-factor game, Corollary of
		Theorem~1.1 in~\cite{liebenau2022threshold}]
		\label{lem:k5}
		There exist constants $c,C>0$ such that the following holds for every large enough integer 
		$n$ divisible by 5. 
		Playing a $(1:b)$ Maker-Breaker game on $K_n$,
		with $b\leq cn^{2/7}$, Maker has a strategy to occupy a 
		spanning $K_5$-factor of $K_n$ within at most $Cn^{12/7}$ rounds.
	\end{lemma}
	
	\begin{proof}
		Let $n$ be large enough. By Theorem~1.1 in~\cite{liebenau2022threshold} there is a 
		constant $c>0$ such that
		Maker can occupy a spanning $K_5$-factor of $K_n$
		against a bias $b^*= \lceil cn^{2/7} \rceil$.
		Let $C=\frac{1}{c}$. By the trick of fake moves
		(see e.g.~Lemma 2.4 in~\cite{clemens2012fast}) 
		it follows that Maker can occupy a spanning $K_5$-factor of 
		$K_n$ against any bias $b\leq b^*$
		within $\lceil \binom{n}{2}/(b^*+1) \rceil \leq C n^{12/7}$
		rounds.
	\end{proof}
	
	\medskip

	\subsection{Waiter-Client game tools}
	When describing strategies for Waiter we will make use of
	the following variant of the Erd\H{o}s-Selfridge Criterion.
	
	\begin{theorem}[Corollary 1.4 in~\cite{bednarska2013weight}]
		\label{thm:WC_transversal}
		Consider a $(1:1)$ Waiter-Client game on a hypergraph $(X,\mathcal{F})$ satisfying
		$$\sum_{F\in \mathcal{F}} 2^{-|F| + 1} < 1.$$
		Then Waiter has a strategy to force Client to
		claim at least one element in each of the
		winning sets in $\mathcal{F}$. 
	\end{theorem}
	
	Moreover, we will use that Waiter has a strategy to force
	large pair degrees. 
	
	\begin{lemma} \label{lem:large.common.degree}
		If $\beta\in (0,1)$, then for every large enough integer $n$ 
		the following holds. 
		Let $G$ be a graph on $n$ vertices such that
		for every two vertices $v,w\in V(G)$ 
		there is a set $N_{v,w}$ of at least 
		$\beta n$ common neighbors.
		Playing a $(1:1)$ Waiter-Client game on $G$, Waiter can force Client to claim a graph 
		$C$ that satisfies the following:
		$$
		|N_C(v)\cap N_C(w)\cap N_{v,w}| 
		\geq \frac{\beta n}{500}  \quad 
		\text{ for every }v,w\subset V(G).
		$$
	\end{lemma}
	
	\begin{proof}
		Before the game starts, we split the edge set of the
		graph $G$ in order to obtain two graphs $G_1$ and $G_2$ such that
		$$
		|N_{G_1}(v)\cap N_{G_2}(w)\cap N_{v,w}| 
		\geq \frac{\beta n}{5} \quad 
		\text{ for every }v,w\subset V(G).
		$$
		That this is possible can be proven
		by taking a partition $G=G_1\cup G_2$ 
		uniformly at random and then showing 
		with the help of a standard Chernoff argument
		(Lemma~\ref{lem:Chernoff}) that the 
		above holds a.a.s.
		
		\smallskip
		
		Then, for a Stage I, Waiter plays on $G_1$ 
		considering the family
		$$
		\mathcal{F}_1 :=
		\left\{
		A:~
		\begin{array}{l}
			A\subset E_{G_1}(v,N_{G_1}(v)\cap N_{G_2}(w)\cap N_{v,w})
			\text{ for some distinct }v,w\in V(G)\\
			\text{and }|A|=0.9|N_{G_1}(v)\cap N_{G_2}(w)\cap N_{v,w}|
		\end{array}
		\right\}.
		$$
		It holds that
		\begin{align*}
			\sum_{F\in \mathcal{F}_1} 2^{-|F|}
			& \leq \sum_{v,w} 
			\binom{|N_{G_1}(v)\cap N_{G_2}(w)\cap N_{v,w}|}{%
				0.1|N_{G_1}(v)\cap N_{G_2}(w)\cap N_{v,w}|}	
			\cdot 2^{-0.9|N_{G_1}(v)\cap N_{G_2}(w)\cap N_{v,w}|}\\
			& \leq \sum_{v,w} 
			(10e)^{0.1|N_{G_1}(v)\cap N_{G_2}(w)\cap N_{v,w}|}	
			\cdot 2^{-0.9|N_{G_1}(v)\cap N_{G_2}(w)\cap N_{v,w}|}\\
			& < \sum_{v,w} 
			0.8^{|N_{G_1}(v)\cap N_{G_2}(w)\cap N_{v,w}|}	
			\leq n^2 0.8^{0.2\beta n} = o(1).
		\end{align*}
		Thus, by Lemma~\ref{thm:WC_transversal}, Waiter can ensure that
		Client claims an element in each set of $\mathcal{F}_1$.
		By this, it follows that Client's subgraph $C_1\subset G_1$ 
		at the end of Stage I satisfies
		$$
		|N_{C_1}(v)\cap N_{G_2}(w)\cap N_{v,w}| 
		\geq \frac{\beta n}{50} \quad 
		\text{ for every }v,w\subset V(G).
		$$
		
		\smallskip
		
		Afterwards, for a Stage II, Waiter plays on $G_2$ 
		considering the family
		$$
		\mathcal{F}_2 :=
		\left\{
		A:~
		\begin{array}{l}
			A\subset E_{G_2}(w,N_{C_1}(v)\cap N_{G_2}(w)\cap N_{v,w})
			\text{ for some distinct }v,w\in V(G)\\
			\text{and }|A|=0.9|N_{C_1}(v)\cap N_{G_2}(w)\cap N_{v,w}|
		\end{array}
		\right\}.
		$$
		Note that so far, no edge of $G_2$ was claimed.
		It analogously holds that
		$\sum_{F\in \mathcal{F}_2} 2^{-|F|} = o(1)$.
		Thus, by Lemma~\ref{thm:WC_transversal}, Waiter can ensure that
		Client claims an element in each set of $\mathcal{F}_2$.
		By this, it follows that Client's subgraph $C_2\subset G_2$ 
		at the end of Stage II satisfies
		$$
		|N_{C_1}(v)\cap N_{C_2}(w)\cap N_{v,w}| 
		\geq \frac{\beta n}{500} \quad 
		\text{ for every }v,w\subset V(G). 
		$$
		Hence, the statement is proven.
	\end{proof}
	
	Finally, we will use that Waiter can force a perfect matching on $K_{5,5}$. Indeed, the statement below is an easy exercise, and it also follows from Stage~II in the proof of Theorem~2.1 in~\cite{clemens2020fast}.
	
	\begin{lemma}[WC Perfect Matching,~\cite{clemens2020fast}]\label{lem:matching.K55}
		Playing a $(1:1)$ Waiter-Client game on $K_{5,5}$,
		Waiter has a strategy to force a perfect matching
		of $K_{5,5}$.
	\end{lemma}

	\subsection{Structural properties of trees} 
	When we want to embed spanning trees into some graph, 
	we may first care about a small subtree with 
	suitable properties. For this, the following lemmas
	will turn out to be useful.
	
	\begin{lemma}[Small subtree lemma] \label{lem:small.subtree1}
		Let $k\in\mathbb{N}$ and let $T$ be a tree on $n\geq 2k$ 
		vertices. Then there exists a set-cover $V(T)=V_A\cup V_B$ 
		such that $T[V_A]$ and $T[V_B]$ are trees,
		$|V_A\cap V_B|\leq 1$, and 
		$k\leq |V_A| < 2k$.
	\end{lemma}
	
	\begin{proof}
		Fix any vertex $r\in V(T)$ as the root of $T$ and orient the edges 
		of $T$ such that every vertex except the root has exactly one 
		ingoing edge. Moreover, for each vertex $v\in V(T)$ 
		denote with $T_v = (V_v,E_v)$ the tree which is induced by all 
		the vertices which can be reached from $v$ by a directed path.
		Now, choose a vertex $w\in V(T)$ such that $T_w$ is a smallest tree
		among all trees $T_v$ with at least $2k$ vertices. 
		Such a tree must exist, as by assumption $|V_r|=n\geq 2k$. 
		Let $w_1,\ldots,w_t$ be all outgoing neighbors of $w$. 
		For each $i\in [t]$ we have that $T_{w_i} \subset T_w$ and hence, 	
		by the choice of $T_w$, we conclude that $|V_{w_i}|<2k$.
		If there exists $i\in [t]$ such that $|V_{w_i}|\geq k$,
		then we can set $V_A:=V_{w_i}$ and $V_B:=V(T)\setminus V_{w_i}$. 
		Otherwise, we have $|V_{w_i}| < k$ for every $i\in [t]$.
		Then let $\ell$ be the smallest integer such that
		$1 + \sum_{i\in [\ell]} |V_{w_i}| \geq k$.
		Such an $\ell$ must exist, as $|V_w|\geq 2k$.
		Moreover, by the sizes of the subtrees $T_{w_i}$, we know that 
		$1 + \sum_{i\in [\ell]} |V_{w_i}| < 2k$. That is, we can choose
		$V_A:=\{w\}\cup \bigcup_{i\in [\ell]} V_{w_i}$ and
		$V_B:=(V(T)\setminus V_A) \cup \{w\}$.
	\end{proof}

	\medskip
	
	\begin{lemma}[Small subtree cover lemma]
		\label{lem:subtree.partition}
		Let $k\in\mathbb{N}$ and let $T$ be a tree.
		Then there exists a set-cover 
		$V(T)=V_1\cup V_2\cup \ldots \cup V_t$ with 
		$t\leq \lceil \frac{v(T)}{k-1} \rceil + 1$ 
		such that the following holds:
		\begin{itemize}
			\item[(i)] $T[V_i]$ is a tree for every $i\in [t]$.
			\item[(ii)] $|V_i|<2k$ for every $i\in t$.
		\end{itemize}
	\end{lemma}
	
	\begin{proof}
		We do an induction on $v(T)$.
		If $v(T)<2k$, there is nothing to do, as we can set 
		$t=1$ and $V_1=V(T)$. So, let $v(T)\geq 2k$. 
		Then by Lemma~\ref{lem:small.subtree1} we can find 
		a set-cover $V(T)=V_A\cup V_B$ such that 
		$T[V_A],T[V_B]$ are trees, $|V_A\cap V_B|\leq 1$, 
		and $k\leq |V_A| < 2k$. In particular, $|V_B|\leq v(T)-k+1$.
		We set $V_1:=V_A$ and by induction 
		we can find a set-cover $V_B=V_2\cup \ldots \cup V_t$ 
		with
		$t\leq \left( \lceil \frac{v(T)-k+1}{k-1} \rceil + 1\right) + 1
		= \lceil \frac{v(T)}{k-1} \rceil + 1$
		and such that $T[V_i]$ is a tree with $|V_i|<2k$ 
		for every $i\in\{2,3,\ldots,t\}$.
		Putting everything together, we obtain a set-cover
		$V(T)=V_1\cup V_2\cup \ldots \cup V_t$ as required.
	\end{proof}

	\begin{lemma}[Lemma 2.1~in~\cite{krivelevich2010embedding}]
		\label{lem:leaves.and.paths.kriv}
		Let $k,\ell,n>0$ be integers.
		Let $T$ be a tree on $n$ vertices with at most $k$ leaves.
		Then $T$ contains a collection of at least
		$\frac{n-(2k-2)(\ell+1)}{\ell+1}$ 
		vertex-disjoint bare paths of length $\ell$.
	\end{lemma}
	
	\begin{corollary}
		\label{cor:leaves.and.paths}
		Let $\ell$ be a positive integer.
		Then there exists a constant $\gamma'>0$
		such that the following holds.
		Every tree $T$ has at least $ \gamma' v(T) $ leaves
		or a collection of at least $ \gamma' v(T) $
		vertex-disjoint bare paths of length $\ell$ each.
	\end{corollary}
	
	\begin{proof}
		Set $\gamma'=\frac{1}{4(\ell +1)}$. Let $T$ be any tree.
		If the number of leaves in $T$ is $k<\gamma'v(T)$, then by 
		Lemma~\ref{lem:leaves.and.paths.kriv} there are at least	
		$$
		\frac{v(T)-(2k-2)(\ell+1)}{\ell+1}
		> \frac{v(T)}{\ell+1} - 2k
		> \left(\frac{1}{\ell+1} - 2\gamma'\right)v(T)
		> \gamma'v(T)	
		$$ 
		bare paths of length $\ell$.
	\end{proof}
	
	\smallskip

	\begin{lemma}[Classifying trees lemma]\label{lem:classify.trees}
		For every $\ell \in \mathbb{N}$, $\delta\in (0,1)$ and 
		$C \in \mathbb{N}$ there exist constants $\gamma, c\in (0,1)$
		such that the following is true for every large enough $n$.
		Let $T$ be a tree on $n$ vertices with maximum degree 
		$\Delta(T) \leq \frac{cn}{\log(n)}$.
		Let $L(T)$ denote the set of leaves of $T$. 
		Then at least one of the following properties hold:
		\begin{enumerate}
			\item[(i)] $T$ has at least $\gamma n$ vertex-disjoint 
			bare paths of length $\ell$.
			\item[(ii)] $|L(T)| \geq C \gamma n$ and 
			there is a tree $T'\subset T$ with
			$v(T')\leq \delta n$ and 
			$|V(T')\cap N_T(L(T))| \geq C \log(n)$.
		\end{enumerate}
	\end{lemma}

	\begin{proof}
		Having $\ell,\delta$ and $C$ fixed, we first choose 
		$\gamma' = \gamma'(\ell)$ according to 
		Corollary~\ref{cor:leaves.and.paths}. We then set
		$c = \frac{\gamma'\delta}{10C}$ and $\gamma = \frac{\gamma'}{2C}$.
		Now, let $T$ be a tree on $n$ vertices with maximum degree 
		$\Delta(T) \leq \frac{cn}{\log(n)}$. Provided $n$ is large enough, 
		we show that if $T$ does not satisfy (i), 
		then property (ii) must hold. 
		
		Since (i) does not hold and $\gamma' > \gamma$, there is 
		no collection of at least $\gamma' n$ vertex-disjoint bare paths 
		of length $\ell$ in $T$. Hence, by 
		Corollary~\ref{cor:leaves.and.paths} we can conclude 
		$|L(T)| \geq \gamma' n$, which in particular gives 
		$$
		|L(T)|\geq C\gamma n
		\quad \text{and} \quad
		|N_T(L(T))| \geq \frac{|L(T)|}{\Delta(T)} 
		\geq \gamma' c^{-1} \log(n),
		$$
		since $\Delta(T)\leq \frac{cn}{\log(n)}$.
		Now, we apply Lemma~\ref{lem:subtree.partition} 
		to the tree $T$ with $k:=\lfloor 0.4\delta n \rfloor$, 
		and we find a set-cover
		$V(T)=V_1\cup V_2\cup \ldots \cup V_t$ such that 
		$t\leq \lceil \frac{n}{k-1} \rceil +1 < 5\delta^{-1}$, 
		and $T[V_i]$ is a tree with
		$|V_i| < 2k < \delta n$ for all $i\in [t]$. 
		By the Pigeonhole Principle there must exist $i^\ast\in [t]$
		such that
		$$
		|V_{i^\ast}\cap N_T(L(T))|\geq \frac{|N_T(L(T))|}{t}
		\geq  0.2 \gamma' c^{-1} \delta \log(n)
		\geq C\log(n)
		$$
		by the choice of $c$. Hence, (ii) follows by setting 
		$T':=T[V_{i^\ast}]$.
	\end{proof}

	\medskip
	
	\subsection{Tree embedding lemmas} 
	For the embedding of almost spanning trees, we may use
	the following variant of an embedding result due to 
	Haxell~\cite{haxell2001tree}.

	\begin{lemma}[Embedding almost spanning trees; 
		variant of Theorem 1 in~\cite{haxell2001tree}]
		\label{lem:haxell}
		Let $T$ be a tree with maximum degree $d$, 
		and let $S\subset T$ be a subtree of $T$. Moreover, 
		let $G$ be a graph and let $g:V(S)\rightarrow V(G)$ be an 
		embedding of $S$ into $G$. Assume that the following properties 
		hold for some $k \in \mathbb{N}$:
		\begin{enumerate}
			\item[(P1)] $|N_G(X)\setminus g(V(S))| \geq d|X| + 1$ 
			for every $X\subset V(G)$ with $1\leq |X|\leq 2k$,
			\item[(P2)] $|N_G(X)| \geq d|X| + v(T)$ for every
			$X\subset V(G)$ with $k<|X|\leq 2k$.
		\end{enumerate}
		Then the embedding $g$ can be extended to an embedding of $T$ 
		into $G$.
	\end{lemma}
	
	\begin{proof}[Sketch of proof.]
		Let $T_0=\varnothing\subset T_1\subset T$ such that 
		$S\subset T_1$ and $T_1$ is a tree which can be obtained from $T$ 
		by removing leaves only. We first note that (P1) and (P2) 
		imply the properties (0)--(2) from Theorem 1 in~\cite{haxell2001tree}, 
		when we set $\ell:=1$, $d_1:=d$.
		Moreover, (P1) implies that
		\begin{align*}
			|N_G(X)\setminus g(V(S))| 
			& \geq d_1|X\cap g(V(S))| + d_1|X\setminus g(V(S))| \\
			& \geq \sum_{x\in X\cap g(V(S))} \big( d_T(g^{-1}(x)) -
			d_S(g^{-1}(x) \big)  + d_1|X\setminus g(V(S))| 
		\end{align*}
		holds for every $X\subset V(G)$ such that $|X|\leq 2k$.
		Because of this, in the proof of Theorem 1 
		in~\cite{haxell2001tree}, the embedding $g$ would be called 
		a Type-1 embedding of $S$ into $G$. Now, let 
		$S=:S_0\subset S_1\subset S_2\subset \ldots \subset S_r:=T_1$ be 
		any sequence of trees such that $S_{i+1}$ is obtained from $S_i$ 
		by attaching one new leaf. By Claim 3 in~\cite{haxell2001tree} it 
		follows iteratively that for every $i\in[r]$ we can extend $g$ to 
		a Type-1 embeddings of $S_i$ into $G$.
		Moreover, by Claim 4 in~\cite{haxell2001tree} (applied with $\ell=1$) it follows that 
		the Type-1 embedding of $S_r=T_1$ can be extended to an 
		embedding of $T$ into $G$. Note that the proofs of these claims 
		only use the fact that $G$ satisfies the properties (0)-(2).
	\end{proof}
	
	Finally, we will use the following lemma, which 
	is helpful for finishing 
	the embedding of spanning trees with many leaves.
	This lemma is a consequence of a generalization
	of Hall's Marriage Theorem.
	
	\begin{lemma}[Star matching lemma, Lemma 3.10 in \cite{johannsen2013expanders}]
		\label{lem:star-matching}
		Let $d,m \in \mathbb{N}$ and let $G$ be a graph. Suppose that two 
		disjoint sets $U,W \subset V(G)$ satisfy the following three 
		conditions:
		\begin{enumerate}
			\item[(i)] $|N_G(X) \cap W| \geq d|X|$ for all 
			$X \subseteq U$ with $1 \leq |X| \leq m$,
			\item[(ii)] $e_G(X,Y) > 0$ for all $X \subseteq U$ and 
			$Y \subseteq W$ with $|X| = |Y| \geq m$,
			\item[(iii)] $|N_G(w) \cap U| \geq m$ for all $w \in W$.
		\end{enumerate}
		Then, for every map $k : U \rightarrow \{1,\dots,d\}$ that 	
		satisfies $\sum_{u \in U} k(u) = |W|$, the set $W$ can be 
		partitioned into $|U|$ disjoint subsets $\{W_u\}_{u \in U}$ 
		satisfying $|W_u| = k(u)$ and $W_u \subseteq N_G(u) \cap W$. 
	\end{lemma}
	
	As in \cite{johannsen2013expanders}
	we call the set of edges between the vertices of $U$ and their respective parts in $W$ a star matching.

	\bigskip

	\section{A tree universal graph}\label{sec:tree.universal}
	
	The following theorem provides a sufficient condition
	for a graph to be universal for all trees of almost linear maximum degree. In the next two sections we will then prove that Maker and Waiter have strategies
	to build a graph satisfying such a condition.
	
	\begin{theorem}\label{thm:universal_new}
		Let $\alpha\in (0,1)$, 
		and $C_0>0$ be any constants.
		There exist constants $\gamma',c>0$ such that 
		the following is true
		for every $\gamma\in (0,\gamma')$ and every large enough integer $n$.
		Let $G=(V,E)$ be a graph on $n$ vertices
		such that the following properties hold:
		\begin{enumerate}
			
			\item[(1)] \emph{Partition:} 
			There is a partition $V=V_1\cup V_2$ such that 
			$|V_2| = 500 \lfloor \gamma n \rfloor$.
			
			\item[(2)] \emph{Suitable star:}
			There are a vertex $x^\ast$ and disjoint sets 
			$R^\ast,S^\ast\subset V_1$ 
			such that the following holds:
			\begin{enumerate}
				\item[(a)] $|S^\ast| = \lfloor 25 C_0\log(n) \rfloor$ and
				$S^\ast \subset N_G(x^\ast)$.
				\item[(b)] $|R^\ast|\leq 25$ and for each $v\in R^\ast$
				the following holds: 
				If $v$ is not adjacent with $x^\ast$, 
				then $v$ is adjacent with a vertex $s_v\in S^\ast$,
				such that $s_v\neq s_w$ if $v\neq w$.
				\item[(c)] For all 
				$w\in V\setminus (R^\ast\cup S^\ast)$, we have
				$d_G(w,S^\ast) \geq 2C_0 \log(n)$.
			\end{enumerate}
			
			\item[(3)] \emph{Pair degree conditions:}
			For every $v\in V(G)$ there are at most $\log(n)$ vertices 
			$w\in V(G)$ such that $|N_G(v)\cap N_G(w)\cap V_1| < 
			\alpha n$.
			
			\item[(4)] \emph{Edges between sets:}
			Between every two disjoint sets 
			$A\subset V_1$ and $B\subset V$ of size
			$\lfloor C_0\log(n) \rfloor$ there is an edge in $G$.
			
			\item[(5)] \emph{Suitable clique factor:}
			In $G[V_2]$ there is a collection
			$\mathcal{K}$ of $100\lfloor \gamma n\rfloor $ 
			vertex-disjoint $K_5$-copies such that the following holds:
			\begin{enumerate}
				\item[(a)] There is a partition 
				$\mathcal{K}=\mathcal{K}_{good}\cup \mathcal{K}_{bad}$
				such that $|\mathcal{K}_{bad}|=\lfloor \gamma n \rfloor$.
				\item[(b)] Every vertex $v\in V$ which is not in a clique 
				of $\mathcal{K}_{good}$
				satisfies $d_G(v,V_2)\geq 40\lfloor \gamma n \rfloor$. 
				\item[(c)] For every clique $K\in \mathcal{K}_{good}$
				there are at most $\gamma n$ cliques
				$K'\in \mathcal{K}_{good}$ such that $G$ does not have
				a matching of size 3 between $V(K)$ and $V(K')$.
			\end{enumerate}	
		\end{enumerate}
		Then $G$ contains a copy of every tree $T$ on $n$ vertices and with maximum degree $\Delta(T)\leq \frac{cn}{\log(n)}$.
	\end{theorem}
	
	The overall idea of the proof will be as follows.
	We will distinguish the desired trees
	by their containment of many bare paths
	or many leaves. When a tree has many bare paths,
	we will first embed everything except from the bare paths into $V_1$,
	by using that $G$ has good expanding properties
	which are guaranteed by properties (3) and (4).
	By making use of the clique factor in (5), we will then manage to complete the embedding of $T$.
	Similarly, when caring about trees with many leaves,
	in a first step we will embed everything 
	except from the leaves into $V_1$
	by using properties (3) and (4), 
	and only afterwards we will care about the leaves by applying the star matching lemma.
	In order to succeed with this application,
	we need to do the first embedding step more carefully.
	We do so by distinguishing two cases,
	depending on whether there exists a vertex $x$
	which is adjacent to many neighbors of leaves.
	If such a vertex $x$ exists, we use property (2),
	embed this particular vertex onto $x^\ast$
	and make sure that each vertex in $S^\ast$
	becomes the image of a leaf neighbor. Property (2.c)
	together with the expanding properties
	then help to verify the conditions of the star matching lemma.
	Otherwise, if such a vertex $x^\ast$ does not exists,
	we apply a random embedding argument together with property (3)
	to ensure the properties needed for the star matching lemma.

	\medskip

	\begin{proof}[Proof of Theorem~\ref{thm:universal_new}]
		In the following we prove Theorem~\ref{thm:universal_new}.
		Let $\alpha$ and $C_0$ be given by the 
		statement of Theorem~\ref{thm:universal_new}.
		Choose $\delta := 0.5 \alpha$, and let $\ell := 502$. 
		Let $m := \lfloor C_0\log(n) \rfloor$.
		Choose $C_1 := \max\{100C_0\alpha^{-1},501\}$. 
		Let $c_0$ and $\gamma_0$ be given 
		by Lemma~\ref{lem:classify.trees} with input $\ell$, $\delta$, and
		$C_1$. Further, choose 
		$\gamma' := \min\{10^{-5},\gamma_0\}$, and $\gamma\in (0,\gamma')$, and let $c := \min\{c_0, \frac{\alpha}{10C_0}, 
		\frac{\gamma}{10C_0}\}$. 
		Let $d := \frac{cn}{\log(n)}$.
		In the following, we assume $n$ to be large enough 
		whenever necessary, e.g.~to apply Lemma~\ref{lem:classify.trees}
		with the specified inputs.

		Let $G$ be a graph satisfying the properties (1)--(5)
		from Theorem~\ref{thm:universal_new}.
		We want to show that $G$ contains a copy 
		of every tree $T$ with maximum degree $\Delta(T) \leq d$. 
		Consider any such tree $T$. 
		Because of Lemma~\ref{lem:classify.trees} 
		(with the inputs and outputs above) we know that $T$  
		contains $\gamma_0 n \geq \gamma n$ vertex-disjoint bare paths of 
		length $\ell$, or $T$ has at least $C_1 \gamma_0 n \geq C_1 \gamma n$
		leaves and contains a small subtree $T' \subset T$ with 
		$v(T') \leq \delta n$ and $|V(T') \cap N_T(L(T))|\geq C_1 \log(n)$. 
		In the following, we will show for each of these two cases separately 
		how we can embed $T$ into $G$.
		
		\smallskip
		
		\textbf{Case 1:} $T$ has at least $\gamma n$ vertex-disjoint bare
		paths of length $502$. In this case, roughly speaking, we embed all of $T$ but some
		of the bare paths into $V_1$ using Lemma~\ref{lem:haxell}, and finish
		the embedding by using the clique factor (property (5))
		to embed the bare paths and 
		absorb the left-over vertices of $V_1$.
		
		Let $\mathcal{P}$ be a family of exactly $\lfloor \gamma n \rfloor$ 
		vertex-disjoint bare paths of length $502$. We form a new tree $T_1$ 
		from $T$ as follows: for each path $P \in \mathcal{P}$ we delete the 
		inner vertices of the path and join the endpoints by an edge. Note 
		that $\Delta(T_1) \leq d$ and 
		$v(T_1) = n - 501\lfloor \gamma n \rfloor$.
		We want to embed $T_1$ into $G[V_1]$ using Lemma~\ref{lem:haxell} 
		(with $S$ being the empty graph and $k=m$). 
		To do so, we have to check the following properties:
		\begin{enumerate}
			\item[(P1)] $|N_G(X) \cap V_1| \geq d |X| + 1$ 
			for every $X \subset V_1$ with $1 \leq |X| \leq 2m$.
			\item[(P2)] $|N_G(X) \cap V_1| \geq d |X| + v(T_1)$ 
			for every $X \subset V_1$ with $m < |X| \leq 2m$.
		\end{enumerate}
		
		By property (3) we conclude that 
		$|N_G(v) \cap V_1| \geq \alpha n$ for all $v \in V_1$.
		Hence, for every $X \subset V_1$ with $1 \leq |X| \leq 2m$
		we obtain
		$|N_G(X) \cap V_1| \geq \alpha n - |X| 
		\geq 0.9 \alpha n \geq d\cdot 2m + 1$
		by the choice of $d$ and $m$, and for $n$ large enough. 
		In particular, (P1) holds. 
		
		Now consider any set $X \subset V_1$ with $m < |X| \leq 2m$. 
		Then, by property (4), 
		less than $m$ vertices in $V_1\setminus X$ 
		are not in the neighborhood of $X$. Therefore,
		\begin{align*}
			|N_G(X) \cap V_1| 
			& > |V_1\setminus X| - m 
			\geq |V_1| - 3m
			= n - 500 \lfloor \gamma n \rfloor - 3m \\
			& > 2 d \cdot m + n - 501 \lfloor \gamma n \rfloor 
			\geq d |X| + v(T_1)
		\end{align*}
		by the choice of $d$ and $m$. That is, (P2) holds.
		As $G$ satisfies (P1) and (P2), we can embed $T_1$ into 
		$G[V_1]$ using Lemma~\ref{lem:haxell}, resulting in an 
		embedding $g : V(T_1) \rightarrow V_1$. Note that this also is an
		embedding of the tree obtained from $T$ by deleting the inner vertices
		of all paths in $\mathcal{P}$.

		\smallskip
		
		Hence, we are left with embedding the $501 \lfloor \gamma n \rfloor$ 
		inner vertices of the family of bare paths $\mathcal{P}$. 
		Let $R := V_1 \setminus g(V(T_1))$ be the set of 
		$\lfloor \gamma n \rfloor$ vertices of $V_1$ 
		which were not used for the embedding of $T_1$.
		For each path $P \in \mathcal{P}$, denote by $v^P$ and $w^P$ 
		the images of the endpoints of $P$ under $g$. Further, since 
		$|\mathcal{P}| = |R| = \lfloor \gamma n \rfloor$, 
		we can fix exactly one distinct
		vertex $r^P \in R$ for every path $P \in \mathcal{P}$. 
		Similarly, since 
		$|\mathcal{P}| = |\mathcal{K}_{bad}| = \lfloor \gamma n \rfloor$, 
		we can fix exactly one distinct
		clique $K^P \in \mathcal{K}_{bad}$ for every path $P \in \mathcal{P}$.
		In each of these cliques we fix two arbitrary vertices
		$x^P,y^P\in V(K^P)$.

		\smallskip
		
		As a first step towards embedding the inner vertices of $\mathcal{P}$,
		we choose a collection of six cliques
		$\mathcal{K}^P = \{K^P_1,\ldots,K^P_6\} \subset \mathcal{K}_{good}$ 
		for every $P \in \mathcal{P}$ such that the following properties 
		are satisfied:
		
		\begin{minipage}[t]{0.49\textwidth}
			\vspace{0pt}
			\begin{enumerate}
				\item[(1)] $\mathcal{K}^P \cap \mathcal{K}^{P'} = \varnothing$ 
				for all $P' \in \mathcal{P}\setminus\{P\}$,
				\item[(2)] $e_G(v^P, V(K^P_1)) > 0$,
				\item[(3)] $e_G(r^P, V(K^P_2)) > 0$,
				\item[(4)] $e_G(r^P, V(K^P_3)) > 0$,
			\end{enumerate} 
		\end{minipage}
		\begin{minipage}[t]{0.49\textwidth}
			\vspace{0pt}
			\begin{enumerate}
				\item[(5)] $e_G(x^P, V(K^P_4)) > 0$,
				\item[(6)] $e_G(y^P, V(K^P_5)) > 0$,
				\item[(7)] $e_G(w^P, V(K^P_6)) > 0$.
			\end{enumerate} 
		\end{minipage}
		
		\medskip
		
		Note that we can find such cliques greedily by property (5.b).
		Indeed, this property ensures that 
		each of the relevant vertices 
		$r^P,v^P,w^P,x^P,y^P$
		is adjacent to
		at least $8\lfloor \gamma n\rfloor$ cliques from $\mathcal{K}$, 
		and hence to at least $7\lfloor \gamma n\rfloor$ cliques from
		$\mathcal{K}_{good}$,
		while for the properties above we only need to choose 
		$6\lfloor \gamma n\rfloor$ cliques in total.
		
		Moreover, based on (2)--(7), let $k_1^P\in V(K_1^P)$
		be a neighbor of $v^P$, let $k_2^P\in V(K_2^P)$
		be a neighbor of $r^P$, and so on, until reaching
		a neighbor $k_6^P\in V(K_6^P)$ of $w^P$.

		\smallskip
		
		Next, we consider the following auxiliary graph $H$: 
		its vertex set is $\mathcal{K}_{good}$ and 
		we put an edge between two vertices $K,K'\in V(H)$
		if and only if in $G$ there is a matching 
		of size $3$ between the cliques $K$ and $K'$.
		For the graph $H'\subset H$ induced on
		$\mathcal{K}'=\mathcal{K}_{good}\setminus 
		\bigcup_{P\in\mathcal{P}} \mathcal{K}^P$
		we then have
		$v(H')=93\lfloor \gamma n\rfloor$
		and $\delta(H')\geq v(H') - \gamma n > v(H')/2$
		by property (5.c).
		Hence, by Dirac's Theorem (see e.g.~\cite{west2001introduction}), we can find a Hamilton cycle in $H'$,
		and we can split this Hamilton cycle
		into $3\lfloor \gamma n\rfloor$ vertex-disjoint paths
		each having exactly 31 vertices. Denote with $\mathcal{P}_H$
		the collection of these paths.
		
		\smallskip
		
		We then define an auxiliary bipartite graph $F=(A\cup B,E(F))$
		with partite sets
		$$A:=\{ (K_i^P,K_{i+1}^P):~ i\in\{1,3,5\}, P\in\mathcal{P}\}
		\quad \text{and} \quad
		B:=\mathcal{P}_H,$$
		where we put an edge between a vertex 
		$(K_i^P,K_{i+1}^P)\in A$ and a vertex $Q\in B$
		if and only if in $H$ there is a perfect matching
		between the endpoints of the path $Q$ and the vertices 
		$K_i^P,K_{i+1}^P$
		(which means that $Q$ can be extended to a longer path in $H$ with endpoints 
		$K_i^P,K_{i+1}^P$).
		We then have
		$|A|=|B|=3\lfloor \gamma n\rfloor$
		and $\delta(F)\geq |A| - \gamma n > |A|/2$. Indeed, a vertex 
		$(K_i^P,K_{i+1}^P)\in A$ and a vertex $Q\in B$ with endpoints $(K_1^Q, K_2^Q)$ are connected in $F$ unless there are at least two edges missing between the vertices $K_i^P,K_{i+1}^P$ and the vertices $K_1^Q, K_2^Q$ in the auxiliary graph $H$. This can happen at most $\gamma n$ times per vertex $v \in V(F)$ since $\delta(H)\geq v(H) - \gamma n$ by (5.c).
		By a standard application of Hall's condition
		(see e.g.~\cite{west2001introduction})
		it follows that $F$ has a perfect matching.
		Denote with $Q_{i,i+1}^P$ the path which is matched to 
		the pair $(K_i^P,K_{i+1}^P)$ in this matching, for every $i\in\{1,3,5\}$ and  $P\in\mathcal{P}$.
		
		\begin{center}
			\includegraphics[scale=0.8,page=1]{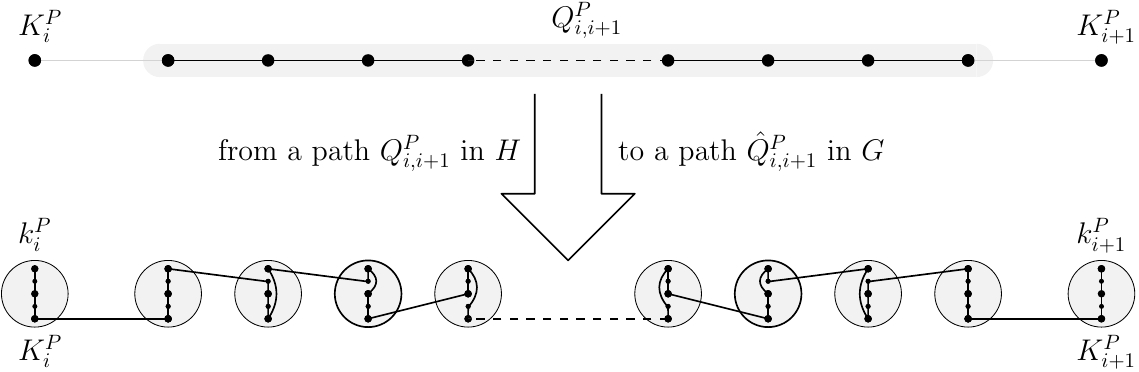}
		\end{center}
		
		Note that this path $Q_{i,i+1}^P$ describes a sequence
		of 31 cliques from $\mathcal{K}_{good}$,
		such that between $K_i^P$ and the first clique in the sequence,
		between two consecutive cliques in the sequence, 
		and between the last clique in the sequence and $K_{i+1}^P$ 
		there is a matching of size $3$ in the graph $G$.

		It is then easy to find a path $\hat{Q}_{i,i+1}^P$
		on 165 vertices which has endpoints $k_i^P$ and $k_{i+1}^P$,
		and which goes through all the vertices of $K_i^P$, 
		$K_{i+1}^P$ and all the cliques
		in $Q_{i,i+1}^P$. Indeed, from each of the mentioned
		matchings pick greedily one edge such that all these edges 
		are independent and not incident with 
		$k_i^P$ or $k_{i+1}^P$;
		then connect $k_i^P$, $k_{i+1}^P$ and these matching edges
		with paths of length 4 that are fully contained in
		one of the relevant cliques.

		\smallskip
		
		We now have everything that we need to describe
		how we can embed the paths $P\in \mathcal{P}$
		into $R\cup V_2$ (with fixed endpoints $v^P,w^P$),
		and thus finish the embedding of $T$.
		Given $P\in \mathcal{P}$, let
		$Q_{xy}^P$ be any path of length 4
		in $K^P$ with endpoints $x^P$ and $y^P$.
		Then we embed $P$ to the path
		given by the sequence
		$(v^P,\hat{Q}_{1,2}^P,r^P,\hat{Q}_{3,4}^P,Q_{xy}^P,\hat{Q}_{5,6}^P,w^P)$.
		Note that this way, the inner vertices of $P$ are embedded
		into $R\cup V_2$, disjointly from the images of all other paths in $\mathcal{P}$.

		\begin{center}
			\includegraphics[scale=0.8,page=2]{path_pic_new}
		\end{center}

		\bigskip
		
		\textbf{Case 2:}
		$T$ has at least $C_1 \gamma n$ leaves and contains a subtree 
		$T' \subset T$ with $v(T') \leq \delta n$ and 
		$|V(T') \cap N_T(L(T))|\geq C_1 \log(n)$.
		Let $N_1\subset V(T') \cap N_T(L(T))$ be any subset
		of size $\lfloor C_1 \log(n) \rfloor$.
		
		In this case, we start by embedding the subtree $T'$ minus the leaves $L(T)$ in such a way that the vertices of $N_1$ are embedded in a suitable way. 
		Afterwards, we extend that embedding 
		to an embedding of $T - L(T)$ by 
		an application of Lemma~\ref{lem:haxell}. 
		Lastly, we use Lemma~\ref{lem:star-matching}, 
		to embed the leaves $L(T)$. 
		More precisely, we set 
		$T_1 := T' - L(T)$,  $T_2 := T - L(T)$,  
		and we distinguish two cases, depending on how the 
		vertices of $N_1$ are distributed in $T_1$.
		
		\textbf{Case 2.1:} Assume that there is a vertex $x\in V(T_1)$ 
		with $d_{T_1}(x,N_1)\geq 0.5C_1\log(n)$.
		Then we do the embedding of $T_1$ as follows:
		we embed $x$ onto $x^\ast$, and we embed the vertices of 
		$N_{T_1}(x,N_1)$ into $N_G(x^\ast)\cap V_1$, which has size at least $\alpha n\geq d_{T_1}(x,N_1)$ by property (3), in such a
		way that all vertices of $(R^\ast\cap N_G(x^\ast))\cup S^\ast$
		are used, which is possible since $d_{T_1}(x,N_1) \geq |R^\ast\cup S^\ast|$ by properties (2.a) and (2.b), and by choice of $C_1$.
		Afterwards, embed the rest of $T_1$ greedily into $V_1$,
		which is possible since by (3) all vertex degrees into $V_1$ are at least $\alpha n > v(T_1)$.
		Because of (2.c), the result then is an embedding $g:V(T_1)\rightarrow V_1$ into $G[V_1]$ such that the following holds:
		every vertex $w\in V(G)\setminus (g(V(T_1))\cup R^\ast)$ satisfies
		$d_G(w,g(N_1))\geq 2C_0\log(n)$.
		
		\smallskip
		
		Next, we extend this to an embedding of $T_2$ into $G[V_1]$.
		For this, we use Lemma~\ref{lem:haxell} (with $S := T_1$, $k := m$
		and with $T$ being replaced with $T_2$).
		To do so, we have to check the following properties:
		\begin{enumerate}
			\item[(P1)] $|(N_G(X) \cap V_1) \setminus g(V(T_1))| \geq d |X| + 1$ 
			for every $X \subset V_1$ with $1 \leq |X| \leq 2m$.
			\item[(P2)] $|N_G(X) \cap V_1| \geq d |X| + v(T_2)$ 
			for every $X \subset V_1$ with $m < |X| \leq 2m$.
		\end{enumerate}
		By property (3) we know that 
		$|(N_G(v) \cap V_1)\setminus g(V(T_1))| 
		\geq \alpha n - \delta n \geq 0.5\alpha n$ 
		for all $v \in V_1$. Hence, for every $X \subset V_1$ 
		with $1 \leq |X| \leq 2m$ we obtain
		$|(N_G(X) \cap V_1)\setminus g(V(T_1))| \geq 0.5\alpha n - |X| 
		\geq 0.4 \alpha n \geq d\cdot 2m + 1$,
		by the choice of $d$ and $m$, and for $n$ large enough. 
		In particular, (P1) holds. 
		Now consider any set $X \subset V_1$ with $m < |X| \leq 2m$. 
		Analogously to Case~1, since 
		$v(T_2)\leq n - C_1\gamma n \leq n - 501 \lfloor \gamma n \rfloor $, we get 
		$|N_G(X) \cap V_1| \geq d |X| + v(T_2)$,
		and hence (P2). As $G$ satisfies (P1) and (P2), we can extend
		$g$ to an embedding of $T_2$ into $G[V_1]$. 
		
		\smallskip
		
		We are left with embedding the leaves of $T$. Set 
		$U := g(N_T(L(T)))$ and $W := V \setminus g(V(T_2))$. 
		
		We first embed the vertices of $R^\ast\setminus g(V(T_2))$
		one by one. Let $w\in R^\ast\setminus g(V(T_2))$,
		then $w$ does not belong to $N_G(x^\ast)$, since otherwise
		it would be used for the embedding of $T_1$. With (2.b) it follows that there is a distinct vertex $s_w\in S^\ast\subset g(N_1)$
		such that $ws_w\in E(G)$. Since $g^{-1}(s_w)\in N_1\subset N_T(L(T))$ we can find a leaf $w'$ in $T$ which is adjacent with
		$g^{-1}(s_w)$. We then extend $g$ by embedding $w'$ to $w$.
		Moreover, 
		we then remove $w$ from $W$, and
		if $w'$ was the only leaf at $g^{-1}(s_w)$,
		we remove $s_w$ from $U$.
		Note that in this procedure we delete at most $|R^\ast|$
		vertices from each of the sets $U$ and $W$.
		
		Finally, we want to find a star matching 
		between the updated sets $U$ and $W$ using Lemma~\ref{lem:star-matching} applied with $d$, $m$, and 
		with $k(u)$ being the number of leaves in $T$ which are 
		adjacent with $g^{-1}(u)$ but are still not embedded,
		for every vertex $u \in U$. Then, by extending $g$ such that for every $u\in U$ we embed the remaining $k(u)$ leaves
		which are adjacent with $g^{-1}(u)$
		to the set $W_u$, as given by 
		Lemma~\ref{lem:star-matching}, the embedding will be finished. 
		Hence, it remains to be shown that we can apply 
		Lemma~\ref{lem:star-matching}. 
		We need to verify that the following three conditions hold:
		\begin{enumerate}
			\item[(i)] $|N_G(X) \cap W| \geq d|X|$ 
			for all $X \subseteq U$ with $1 \leq |X| \leq m$.
			\item[(ii)] $e_G(X,Y) > 0$ for all 
			$X \subseteq U$ and $Y \subseteq W$ with $|X| = |Y| \geq m$.
			\item[(iii)] $|N_G(w) \cap U| \geq m$ for all $w \in W$.
		\end{enumerate}
		By property (5.b) and since $U\subset V_1$ and $V_2\setminus R^\ast\subset W$, it follows that 
		$|N_G(X) \cap W| \geq 40\lfloor \gamma n\rfloor - |R^\ast|\geq d m$
		for every $X\subset U$, 
		by the choice of $d$ and $m$, and thus (i) holds.
		By property (4) and since 
		$U\subset V_1$, we can immediately conclude (ii).
		Lastly, (iii) follows directly from the fact that 
		$d_G(w,g(N_1))\geq 2C_0\log(n)$
		for all $w \in V(G)\setminus (g(V(T_1))\cup R^\ast)$,
		and since $R^\ast \cap W = \varnothing$.

		\medskip
		
		\textbf{Case 2.2:} Assume that there is no vertex $x\in V(T_1)$ 
		with $d_{T_1}(x,N_1)\geq 0.5C_1\log(n)$.
		This case works essentially the same way as Case 2.1;
		the main differences will be that we embed $T_1$ in a random way,
		and that we do not need to care separately about the vertices
		of $R^\ast$.
		For the first step we claim the following.
		
		\begin{claim}\label{clm:small.tree.random}
			There is an embedding $g:V(T_1)\rightarrow V_1$ of $T_1$ into $G[V_1]$
			such that every vertex $w\in V(G)\setminus g(V(T_1))$ satisfies
			$d_G(w,g(N_1))\geq 2C_0\log(n)$.
		\end{claim}
		
		Before proving this claim, let us quickly explain how to finish the argument then. Using Lemma~\ref{lem:haxell} analogously to Case 2.1
		we can extend the embedding to the tree $T_2$. Afterwards,
		we embed the leaves of $T$ as follows. We again set
		$U := g(N_T(L(T)))$ and $W := V \setminus g(V(T_2))$, 
		but this time we do not
		embed the vertices of $R^\ast\setminus g(V(T_2))$ separately. 
		In this case, by the above claim,
		we know that even these vertices have degree
		at least $2C_0\log(n)$ into $g(N_1)\subset U$, 
		and hence, without updating $U$ and $W$, 
		the properties (i)--(iii) can be checked as in Case 2.1.
		Thus, we can finish the star matching with appropriate sizes for the stars and finish the embedding of $T$.
		
		\smallskip
		
		Hence, it remains to prove Claim~\ref{clm:small.tree.random}.
		For this, we embed $T_1$ in a random way into $V_1$:
		Let $t:=v(T_1)$, and fix an arbitrary ordering 
		$v_1,v_2,\ldots,v_t$ of the vertices of $T_1$ such that 
		every vertex $v_i$ has exactly one
		neighbor of smaller index; denote this neighbor $v_i^-$.
		Furthermore, set $T_i:=T[\{v_1,\ldots,v_i\}]$.	
		We consider the following simple randomized embedding $g$: 
		Choose $g(v_1)\in V_1$ uniformly at random.
		Then, for $i\in \{2,\ldots, t\}$,
		choose $g(v_i)$ uniformly at random from the
		set $(N_G(g(v_i^-))\cap V_1)\setminus g(V(T_{i-1}))$.
		
		We first observe that surely we succeed in embedding $T$. Indeed,
		in each step of the algorithm we have 
		\begin{align*}
			|(N_G(g(v_i^-))\cap V_1)\setminus g(V(T_{i-1}))|
			\geq |(N_G(g(v_i^-))\cap V_1)| - |g(V(T_{i-1}))| 
			\geq \alpha n - \delta n 
			\geq 0.5\alpha n,
		\end{align*}
		and we thus never run out of candidates for embedding a vertex $v_i$. 
		
		\smallskip
		
		Hence, it remains to check that a.a.s.
		\begin{enumerate}
			\item[(D)] for every $w\in V(G)\setminus g(V(T_1))$ we get
			$d_G(w,g(N_1))\geq 2C_0\log(n)$.
		\end{enumerate}
		In order to do so, let us define $P_1:=\{v^-:~ v\in N_1\}$ as the set of parents of the vertices in $N_1$, and note that $|P_1|\leq |N_1|\leq C_1\log(n)$.
		Moreover, for a vertex $w\in V(G)$ say that $v\in P_1$ is \textit{bad} in the embedding $g$ if $|N_G(g(v))\cap N_G(w)\cap V_1|<\alpha n$.
		We first prove that a.a.s.~the following holds:
		\begin{enumerate}
			\item[(B)] for every vertex $w\in V(G)$ there is at most one bad vertex in $P_1$ throughout the embedding process.
		\end{enumerate}
		Fix any $w\in V(G)$. Whenever we embed a vertex $v\in P_1$,
		we have a candidate set of size at least $0.5\alpha n$,
		as shown above. However, because of property (3), 
		there are at most $\log(n)$ candidates whose choice would make
		$v$ a bad vertex for $w$.
		Hence, the probability that $v$ becomes bad for $w$ is bounded
		by $\frac{2\log(n)}{\alpha n}$. It follows that the probability 
		that at least two vertices in $P_1$ become bad for $w$ is bounded by
		$
		\binom{|P_1|}{2}\cdot \left( \frac{2\log(n)}{\alpha n} \right)^2 < n^{-1.5},
		$
		provided $n$ is large enough. Hence, doing a union bound over all vertices $w\in V(G)$, it follows that (B) fails with probability at most $n\cdot n^{-1.5}=o(1)$.
		
		\smallskip
		
		From now on, let us condition on (B), and prove that (D) a.a.s.~holds. Again, fix any $w\in V(G)$.
		Since  by the assumption of Case 2.2 no vertex in $T_1$ is adjacent with more than $0.5C_1\log(n)$ vertices from $N_1$,
		it follows that the parent of at least $\lfloor 0.5C_1\log(n)\rfloor $ vertices 
		$v\in N_1$ is not bad for $w$ (where these parents do not need to be distinct), e.g.
		$|N_G(g(v^-))\cap N_G(w)\cap V_1|\geq \alpha n$.
		Therefore, whenever we embed one of these vertices 
		from $N_1$ into $V_1$, say it is a vertex $v_i$,
		then the probability that it ends up in $N_G(w)$ is
		at least 
		$$
		\frac{|(N_G(g(v_i^-))\cap N_G(w)\cap V_1)\setminus g(V(T_{i-1}))|}{%
			|(N_G(g(v_i^-))\cap V_1)\setminus g(V(T_{i-1}))|}
		\geq
		\frac{\alpha n - \delta n}{n} \geq 0.5\alpha ,
		$$
		and this bound of $0.5\alpha $ holds independently of
		the embeddings of other vertices from $N_1$.
		Thus, the random variable $X_w$ counting the number of vertices from
		$N_1$ ending up in $N_G(w)$ 
		stochastically dominates the binomial
		random variable
		$X\sim\text{Bin}(\lfloor 0.5C_1\log(n) \rfloor, 0.5\alpha )$
		with expectation about $0.25\alpha C_1\log(n) \geq 4C_0\log(n)$.
		By Lemma~\ref{lem:Chernoff} we conclude
		$$
		\Prob(X_w<2C_0\log(n))
		\leq
		\Prob(X < 0.5 \Exp(X))
		\leq 
		e^{-\frac{1}{8}\Exp(X)}
		<
		e^{-0.03\alpha C_1\log(n)}
		<
		e^{-2\log(n)},
		$$
		by the choice of $C_1$. Now, doing a union bound over all 
		$w\in V(G)$, we see that (D) fails with probability at most
		$ne^{-2\log(n)} = o(1)$.
	\end{proof}

	\bigskip

	\section{Maker's strategy}\label{sec:maker.strategy}
	
	\begin{proof}[Proof of Theorem~\ref{thm:MB.tree.universal}]
		Maker's goal it to occupy a graph $M$ which
		satisfies the properties (1)--(5) from 
		Theorem~\ref{thm:universal_new},
		with $\alpha,\gamma,c,C_0$ being chosen in an 
		appropriate way,
		as this theorem then ensures that Maker claims a
		graph as required.
		
		Choose $\alpha := 10^{-8}$ and $C_0 := 2000$,
		let $\gamma'$ and $c$ be given according to 
		Theorem~\ref{thm:universal_new},
		and let $\gamma := \min\{\gamma',10^{-5}\}$. 
		Whenever necessary, we assume that $n$ is large enough.
		Maker's strategy consists of two main stages
		that split into several subgames in which 
		she cares about the required properties (1)--(5) 
		of Theorem~\ref{thm:universal_new} separately.
		If at any point in the game she is unable to follow her strategy, she 
		forfeits the game (we will later see that this does not happen). In 
		the following, we will first describe the overall strategy. In the 
		strategy discussion we will then show that Maker can 
		follow the proposed strategy and occupy a graph 
		with the properties (1)--(5).
		
		\smallskip
		
		\textbf{Strategy description:} Maker's strategy consists of two main stages
		between which there is an additionally preparatory step in which no move is made,
		but the free edges are partitioned in a suitable way into several subboards.
		
		\medskip	
		
		\textbf{Stage I:} This stage consists of two substages:
		
		\begin{enumerate}
			\item[] {\bf Stage I.a:} Maker chooses an arbitrary vertex $x^*$ and claims 
			edges incident to $x^*$ until $N_M(x^*) = \lfloor 25C_0 \log(n)\rfloor $. 
			Let $S^* := N_M(x^*)$ at the end of this substage.
			\item[] {\bf Stage I.b:}  Afterwards consider all vertices $v \in V$ 
			with $d_B(v,S^*) > C_0 \log(n)$. 
			Let $R^*=\{v_1,v_2,\ldots,v_r\}$ be the union of those vertices. 
			For every $i\in [r]$, in the $i^{\text{th}}$ round of this substage,		
			Maker claims the edge $v_ix^*$ if possible, 
			otherwise she claims an edge $v_is_{v_i}$ such that 
			$s_{v_i}\in S^*$ and $d_M(s_{v_i})=1$. 
			Details are given in the strategy discussion.
		\end{enumerate}
		
		\textbf{Preparatory step:}
		Fix a partition $V = V_1 \cup V_2$ with $|V_2| = 500 \lfloor \gamma n \rfloor$ 
		and such that $K_n[V_2]$ does not contain any edges claimed so far by 
		Maker or Breaker. Moreover, find a partition of the graph induced by
		$E_F(K_n) \setminus E_F(V_2)$ into five graphs
		$G_1 \cup G_2 \cup G_3 \cup G_4 \cup G_5$ such that all of the following properties hold:
		
		\begin{enumerate}
			\item[(G1)] For every $v \in V\setminus (R^\ast\cup S^\ast)$ it holds that $d_{G_1}(v,S^*) > 4 C_0 \log(n)$.
			\item[(G2)] For every $v \in V_1$  it holds that $d_{G_2}(v,V_2) > 80 \gamma n$.
			\item[(G3)] For any two disjoint sets $A \subset V_1, B \subset V$ of sizes $|A| = |B| = \lfloor C_0 \log(n) \rfloor$ it holds that $e_{G_3}(A,B) \geq 0.1 C_0^2 (\log(n))^2$.
			\item[(G4)] For every $v \in V$ it holds that $d_{G_4}(v,V_1) > \frac{n}{10}$.
			\item[(G5)] For any two disjoint sets $A\subset V_1,B \subset V$ of sizes $|A| = \frac{n}{40}, |B| = \lfloor \log(n) \rfloor$ it holds that $e_{G_5}(A,B) > \frac{n \log(n)}{250}$.
		\end{enumerate}
		Details on why such a partition exists are given in the strategy discussion.
		
		\medskip
		
		\textbf{Stage II:} We split Stage II into six subgames which are played simultaneously on disjoint boards. During this stage, whenever Breaker claims an edge in one of the boards, Maker reacts on the same board by claiming one edge according to the correspondent strategy. Only in case that there is no free edge
		left of the relevant board, Maker claims an arbitrary free edge of another board.
		
		Maker's boards and goals in these subgames are described in the following. 
		
		\begin{enumerate}
			\item[] {\bf Subgame 1:} Playing on $G_1$, Maker ensures that by the end of the game for every $v \in  V\setminus (R^\ast\cup S^\ast)$ it holds that $d_M(v,S^*) \geq 2 C_0 \log(n)$.
			\item[] {\bf Subgame 2:} Playing on $G_2$, Maker ensures that by the end of the game for every $v \in V_1$ it holds that $d_M(v,V_2) \geq 40 \gamma n$.
			\item[] {\bf Subgame 3:} Playing on $G_3$, Maker claims at least one edge between any two disjoint sets $A \subset V_1$ and $B \subset V$ of sizes $|A| = |B| = \lfloor C_0 \log(n) \rfloor$.
			\item[] {\bf Subgame 4:} Playing on $G_4$, Maker ensures that by the end of the game for every $v \in V$ it holds that $d_M(v,V_1) \geq \frac{n}{40}$.
			\item[] {\bf Subgame 5:} Playing on $G_5$, Maker ensures that by the end of the game for any two disjoint sets $A\subset V_1,B \subset V$ of sizes $|A| = \frac{n}{40}, |B| = \lfloor \log(n) \rfloor $ it holds that $e_M(A,B) > \frac{n \log(n)}{2500}$.
			\item[] {\bf Subgame 6:} Playing on $E(V_2)$, Maker considers two substages. 
			
			\textit{Substage I:} Within at most $n^{1.8}$ rounds on $E(V_2)$, 
			Maker creates a $K_5$-factor on $V_2$ and simultaneously makes sure 
			that for every $v \in V_2$ it holds that 
			$d_B(v,V_2) \leq 0.4|V_2| + d_M(v,V_2)$. 
			Let $\mathcal{K}$ be the collection of $K_5$-copies that form the
			$K_5$-factor at the end of this substage, 
			let $\mathcal{K}_{bad}$ denote a subset of $\lfloor \gamma n \rfloor$ 
			of these $K_5$-copies with the most adjacent Breaker edges within $V_2$,
			and let $\mathcal{K}_{good}=\mathcal{K}\setminus \mathcal{K}_{bad}$.
			
			\textit{Substage II:} Maker makes sure that by the end of the game
			\begin{enumerate}
				\item[(i)]	every vertex $v$ which belongs to a clique in 
				$\mathcal{K}_{bad}$
				satisfies $d_M(v,V_2)\geq 40\lfloor \gamma n \rfloor$, 
				\item[(ii)] for every clique $K\in \mathcal{K}_{good}$
				there are at most $\gamma n$ cliques
				$K'\in \mathcal{K}_{good}$ such that $M$ does not have
				a matching of size 3 between $V(K)$ and $V(K')$.
			\end{enumerate}
		\end{enumerate}

		The details on how Maker can achieve these goals will be given in the 
		strategy discussion.
		
		\smallskip
		
		Before coming to the strategy discussion let us first check that,
		if Maker can follow the strategy without forfeiting the game 
		and can reach all the described goals,
		her graph $M$ fulfils the properties (1)--(5) of 
		Theorem~\ref{thm:universal_new} by the end of the game. 
		Property~(1) is true by the partition from the preparatory step.
		(2.a) is ensured in Stage I.a.
		For property (2.b) note that $|R^\ast|\leq 25$, by the definition
		of $R^\ast$ and since Stage I.a lasts $\lfloor 25C_0\log(n)\rfloor$ rounds.
		The rest of (2.b) follows from the goal of Stage I.b.
		Moreover, (2.c) is ensured in the Subgame 1.
		Property (3) is given by the following reason: For any $v\in V$
		we have $d_M(v,V_1)>\frac{n}{40}$ by Subgame~4, and hence, by the outcome
		of Subgame 5, there are less than $\log(n)$ vertices 
		which have at most $\alpha n$ neighbours into $N_M(v)\cap V_1$.
		Property (4) is obtained in the Subgame 3. 
		For property (5.a) note that in the first substage of Subgame 6, 
		we obtain $\mathcal{K}$, and a partition 
		$\mathcal{K}=\mathcal{K}_{good}\cup \mathcal{K}_{bad}$ as desired.
		Property (5.b) follows from the outcome of Subgame 2 and
		(i) in Subgame 6; property (5.c) is ensured by (ii) in Subgame 6.
		
		\smallskip
		
		\textbf{Strategy discussion:}
		We discuss all of the stages separately.
		
		\textbf{Stage I:} Maker can clearly follow Stage I.a, provided $n$ is
		large enough. So, consider Stage I.b from now on.
		Note that $e(B)=|S^\ast|$ and $d_M(s)=1$ for every $s\in S^\ast$
		when Maker enters Stage I.b.
		If $r=|R^\ast|=1$, then between the unique vertex $v_1\in R^\ast$
		and the set $\{x^\ast\}\cup S^\ast$ there must be at least one
		free edge. Hence, Maker can play as suggested.
		If otherwise $r>1$ then, at the beginning of this stage,
		every vertex $v_i\in R^\ast$ satisfies
		$d_B(v_i,S^\ast)<e(B) - C_0\log(n) = |S^\ast| - C_0\log(n),$
		and hence, Maker in each of the at most 25 rounds of Stage I.b 
		can find a vertex $s_{v_i}\in S^\ast$ as described by the strategy
		and such that $v_is_{v_i}$ is still free,
		if $v_ix^\ast$ is already blocked by Breaker.
		
		\smallskip
		
		\textbf{Preparatory step:}
		Provided $n$ is large enough, it is clear that we can find a partition 
		$V = V_1 \cup V_2$ with $|V_2| = 500 \lfloor \gamma n \rfloor$ 
		and such that $K_n[V_2]$ does not contain any edges claimed by 
		Maker or Breaker during Stage I, since so far at most $25 C_0 \log(n) + 25$ 
		rounds have been played and hence at most $51C_0\log(n)$ edges are claimed. 
		In order to show that there is a partition 
		$G_1\cup G_2\cup G_3\cup G_4\cup G_5$
		of the graph induced by 
		$E_F(K_n) \setminus E(V_2)$  
		as desired, we show that a randomly chosen partition 
		$G_1\cup G_2\cup G_3\cup G_4\cup G_5$
		a.a.s.~satisfies the properties (G1)--(G5).
		To be more precise, for each edge $e \in E_F(K_n) \setminus E(V_2)$ 
		we decide independently with equal probability $1/5$ in which of the graphs 
		$G_1,G_2,G_3,G_4,G_5$ it will be included.
		We consider each of the desired properties separately.
		
		\begin{itemize}
			\item[(G1)] Let $v \in V\setminus(S^* \cup R^*)$, then $d_M(v,S^\ast)=0$
			at the end of Stage I. Moreover, we have $d_B(v,S^\ast)\leq C_0\log(n) + 25$,
			as $v\notin R^\ast$ and Stage I.b lasts at most 25 rounds.
			Hence $d_F(v,S^\ast)>23C_0\log(n)$.
			For the random variable $X^{G_1}_v = d_{G_1}(v,S^*)$
			we have $X^{G_1}_v \sim \text{Bin} \left(d_F(v,S^\ast), \frac{1}{5} \right)$ 
			with expectation $\mathbb E(X^{G_1}_v) = 0.2d_F(v,S^\ast) > 4.6C_0\log(n)$. 
			Applying Chernoff (Lemma~\ref{lem:Chernoff}) we find that 
			$\mathbb{P}\left(d_{G_1}(v,S^*) < 4 C_0 \log(n)\right) 
			< \exp \left( - 0.05 C_0 \log(n) \right) $. 
			With a union bound over all $v \in V\setminus(S^* \cup R^*)$ 
			we see that (G1) fails with probability at most 
			$n \exp (-0.05 C_0 \log(n)) = o(1)$, by the choice of $C_0$.	
			
			\item[(G2)] Let $v\in V_1$. At the end of 
			Stage I we have $d_F(v,V_2)>|V_2|-e(B)>495\gamma n$, provided $n$ is large.
			Hence, for the random variable $X^{G_2}_v = d_{G_2}(v,V_2)$
			we have $X^{G_2}_v \sim \text{Bin} \left(d_F(v,V_2), \frac{1}{5} \right)$  
			with expectation
			$\mathbb E(X^{G_2}_v) > 99 \gamma n $. 
			Applying Chernoff (Lemma~\ref{lem:Chernoff}) and union bound as before,
			we see that (G2) fails with probability $o(1)$.
			
			\item[(G3)] Let $A \subset V_1$ and $B \subset V$ be disjoint sets
			of size $\lfloor C_0\log(n) \rfloor$. Then
			$e_F(A,B)\geq |A|\cdot|B| - 51C_0\log(n) > 0.99 C_0^2(\log(n))^2$,
			provided $n$ is large enough.
			For the random variable $X^{G_3}_{A,B} = e_{G_3}(A,B)$
			we have $X^{G_3}_{A,B} \sim \text{Bin} \left(e_F(A,B), \frac{1}{5} \right)$ 	
			with expectation $\frac15 e_F(A,B) \geq 0.19 C_0^2(\log(n))^2$. 
			Applying Chernoff (Lemma~\ref{lem:Chernoff}) and a union bound over all 
			possible pairs of sets $A,B$ we find that (G3) fails with probability $o(1)$.
			
			\item[(G4)] This can be verified analogously to (G2), using that
			for every $v\in V$ we have
			$d_F(v,V_1)\geq |V_1| - 51C_0\log(n) \geq n - 501 \gamma n > 0.9 n$,
			provided $n$ is large enough.
			
			\item[(G5)] This can be verified analogously to (G3), using that
			for every $A\subset V_1$, $B\subset V$ of sizes $|A|=\frac{n}{40}$
			and $|B|= \lfloor \log(n) \rfloor$ we have
			$d_F(A,B)\geq |A|\cdot |B| - 51C_0\log(n) \geq \frac{n \log(n)}{41}$,
			provided $n$ is large enough.
		\end{itemize}
		
		\textbf{Stage II:} We discuss the 6 subgames separately.
		
		\begin{enumerate}
			\item[] \textbf{Subgame 1:}
			Maker can reach her goal by (G1) and simply using a pairing strategy
			for the edges in $E_{G_1}(v,S^\ast)$ for every vertex 
			$v\in V\setminus (R^\ast\cup S^\ast)$.
			\item[] \textbf{Subgame 2:}
			Maker can reach her goal by (G2) and simply using a pairing strategy
			for the edges in $E_{G_2}(v,V_2)$ for every vertex 
			$v\in V_1$.
			\item[] \textbf{Subgame 3:}
			Maker can reach her goal by (G3) and the Erd\H{o}s-Selfridge-Criterion 
			(see Lemma~\ref{lem:ES-criterion}) as follows: Consider the family
			$$\mathcal{F}_3 := \{E_{G_3}(A,B): A \subset V_1, B \subset V,
			|A|=|B| = \lfloor C_0 \log(n) \rfloor, A \text{ and } B \text{ disjoint} \},$$
			and note that for large enough $n$, using (G3), we get
			$$
			\sum_{F \in \mathcal{F}_3} 2^{-|F|+1} \leq
			(n^{C_0\log(n)})^2\cdot 2^{-0.1C_0^2(\log(n))^2+1}
			\leq e^{2 C_0 (\log(n))^2 - 0.1 C_0^2 (\log(n))^2+1} < 1 ,
			$$
			by the choice of $C_0$. Hence Maker (taking the role of Breaker)
			can claim an edge of every edge set in $\mathcal{F}_3$.
			
			\item[] \textbf{Subgame 4:} Let $H_4\subset G_4$ be the graph
			induced by all edges of $G_4$ that intersect $V_1$, and note that
			by (G4), $\delta(H_4)>\frac{n}{10}$. Maker can reach her goal 
			by an application of Lemma~\ref{lem:MinDeg} to $H_4$.
			
			\item[] \textbf{Subgame 5:}
			Maker can reach her goal by (G5) and the criterion 
			in Lemma~\ref{lem:criterion.multistage} as follows: We 
			choose $\delta := \frac15$ and consider the family
			$$\mathcal{F}_5 = \{E_{G_5}(A,B) : 
			A \subset V_1, B \subset V, |A| = \frac{n}{40}, |B| = \lfloor \log(n) \rfloor, A \text{ and } B \text{ disjoint} \}.$$
			Then
			$k := \min_{F \in \mathcal{F}_5} |F| > \frac{n \log(n)}{250} > 100 \log(4^n) > 4\delta^{-2} \log(|\mathcal{F}_5|)$ for large enough $n$. Thus, by
			Lemma~\ref{lem:criterion.multistage}, Maker can claim at least 
			$\frac{3}{10} \cdot \frac{n \log(n)}{250} > \frac{n \log(n)}{2500}$ edges 
			in each set of $\mathcal{F}_5$.
			
			\item[] \textbf{Subgame 6:}
			For Substage I, Maker plays on $K_n[V_2]$
			alternating between the strategies
			from Lemma~\ref{lem:Deg} and Lemma~\ref{lem:k5},
			applied with $b:=2$ (since Breaker claims two
			edges between any two moves of Maker for any of these lemmas). 
			Lemma~\ref{lem:k5} ensures that Maker obtains a $K_5$-factor
			on $V_2$ before $n^{1.8}$ rounds are played on $V_2$.
			Lemma~\ref{lem:Deg} ensures that throughout this stage, $d_B(v,V_2) \leq 0.4|V_2| + d_M(v,V_2)$ holds for every $v\in V_2$, because of the following reason: Assume that Breaker could reach
			$d_B(v,V_2) > 0.4|V_2| + d_M(v,V_2)$ for some vertex $v\in V_2$ at some point, then by claiming only edges in $E_F(v,V_2)$ from that moment on,
			Breaker could maintain the inequality $d_B(v,V_2) > 0.4|V_2| + d_M(v,V_2)$, eventually leading to
			$d_B(v,V_2) \geq 0.7|V_2|$ by the end of the game, contradicting
			Lemma~\ref{lem:Deg} for large enough $n$.

			At the end of Substage I, let 
			$\mathcal{K}=\mathcal{K}_{good}\cup \mathcal{K}_{bad}$
			be given according to the strategy description,
			and let $V_{good}$ and $V_{bad}$ be the vertices 
			of all cliques in $\mathcal{K}_{good}$ and $\mathcal{K}_{bad}$, respectively.
			Moreover, let
			$$\mathcal{K}_{good}^{(2)}:=\{(K,K'): K,K'\in \mathcal{K}_{good},~ K\neq K',~ e_B(V(K),V(K'))=0\}.$$
			Note that, since Breaker has at most $n^{1.8}$ edges
			within $V_2$ at the end of Substage I, 
			and by definition of
			$\mathcal{K}_{bad}$, we have that for every $K\in \mathcal{K}_{good}$
			there are less than $\gamma n$ cliques $K'\in \mathcal{K}_{good}$
			such that $(K,K')\notin \mathcal{K}_{good}^{(2)}$.

			Now, in Substage II, Maker considers the disjoint boards 
			$E_F(V_{good})$ and $E_F(V_{good},V_{bad})$,
			and always makes her move on the same board that Breaker made his previous move on.
			On the first board $E_F(V_{good})$, Maker makes sure to get a matching of 
			size 3 between any pair $(K,K')\in \mathcal{K}_{good}^{(2)}$.
			Note that this is done easily, 
			since no edge between $V(K)$ and $V(K')$ is blocked by
			Breaker yet. This way, property (ii) of Substage II is ensured.
			On the second board $E_F(V_{good},V_{bad})$, Maker plays a degree game as follows:
			Let $v$ belong to a clique in $\mathcal{K}_{bad}$.
			If by the end of Substage I, we already have $d_M(v,V_2)\geq 40\gamma n$,
			then there is nothing to be done. Otherwise, by the outcome of Substage I,
			it holds that 
			$$d_B(v,V_2)\leq 0.4|V_2| + d_M(v,V_2) <240\gamma n.$$ 
			Then, by pairing the edges in $E_F(v,V_{good})$, 
			Maker can ensure to get 
			$d_M(v,V_2)\geq  40\gamma n.$
			Hence, property (i) of Substage II is ensured as well. \qedhere
		\end{enumerate}
	\end{proof}
	
	\bigskip

	\section{Waiter's strategy}\label{sec:waiter.strategy}
	
	\begin{proof}[Proof of Theorem~\ref{thm:WC.tree.universal}]
		Waiter's goal it to force Client to occupy a graph 
		$C$ which
		satisfies the properties (1)--(5) from 
		Theorem~\ref{thm:universal_new},
		with $\alpha,\gamma,c,C_0$ being chosen in an 
		appropriate way,
		as this theorem then ensures that $C$
		contains a copy of every tree $T$ on $n$
		vertices with $\Delta(T) \leq \frac{cn}{\log(n)}$.
		
		Choose $\alpha := 10^{-8}$ and $C_0 := 2000$,
		let $\gamma'$ and $c$ be given according to 
		Theorem~\ref{thm:universal_new},
		and let $\gamma := \min\{\gamma',10^{-5}\}$. 
		Whenever necessary, we assume that $n$ is large enough.
		Waiter's strategy consists of five stages
		in which she cares about the required properties 
		(1)--(5). 
		If at any point in the game she is unable to follow her 
		strategy, she forfeits the game 
		(we will later see that this does not happen). 
		In the following, we will first describe the overall 
		strategy. In the strategy discussion we will then show 
		that Waiter can 
		follow the proposed strategy and force a graph 
		with the properties (1)--(5).
		
		\smallskip	
		
		\textbf{Strategy description:}
		Waiter's overall strategy consists of five stages,
		and a preparatory step between Stage II and Stage III,
		in which no move is made,
		but the free edges are partitioned in a suitable way into several subboards.
		Before the game starts, fix a partition $V=V_1\cup V_2$ such that $|V_2| = 500 \lfloor \gamma n \rfloor$, and fix an equipartition 
		$V_2=V_{2,1}\cup \ldots \cup V_{2,100}$, 
		i.e.~$|V_{2,j}|=5\lfloor \gamma n \rfloor$ for every $j\in [100]$.	
		
		\smallskip
		
		\textbf{Stage I:} Waiter chooses an arbitrary vertex $x^* \in V_1$. Offering only edges in $E_{K_n}(V_1)$ incident to $x^*$ for $\lfloor 25 C_0 \log(n) \rfloor$ turns she creates a star with center $x^*$. Once this is done, let $S^* := N_C(x^*)$ and $R^* = \emptyset$. Afterwards she continues with Stage II.
		
		\smallskip	
		
		\textbf{Stage II:} This stage consists of two substages:
		\begin{itemize}
			\item[] {\bf Stage II.a:} 
			For every $j\in [100]$, Waiter forces a $K_5$-factor 
			on the graph $K_n[V_{2,j}]$. 
			The details will be given in the strategy discussion.
			Having done so, label 
			the $K_5$-copies in these factors with $K_{i,j}$ 
			and $i\in [\lfloor \gamma n \rfloor]$. Let 
			$\mathcal{K} := \{K_{i,j}:~ i\in [\lfloor \gamma n \rfloor],
			~j\in [100]\}$.
			\item[] {\bf Stage II.b:} For every $j_1\neq j_2$ and 
			$i_1,i_2\in [\lfloor \gamma n \rfloor]$, Waiter forces a 	
			perfect matching between $V(K_{i_1,j_1})$ and 
			$V(K_{i_2,j_2})$. Details will be given in the 
			strategy discussion. Afterwards, Waiter 
			does a preparatory step and afterwards proceeds
			with Stage III.
		\end{itemize}	
		
		\smallskip
		
		\textbf{Preparatory step:}
		Fix a partiton of the graph induced by $E_F(K_n)\setminus E_{K_n}(V_2)$
		into $G_1 \cup G_2 \cup G_3 \cup G_4$ such that the following holds:
		\begin{itemize}
			\item[(G1)] For every pair of vertices $v,w\in V$ it holds that 
			$|N_{G_1}(v)\cap N_{G_1}(w)\cap V_1|\geq 0.05 n$.
			\item[(G2)] For every $v\in V_1$ we have 
			$d_{G_2}(v,V_2)>80\lfloor \gamma n \rfloor$.
			\item[(G3)] For every disjoint sets $X\subset V_1,Y\subset V$ 
			of size $|X|=|Y|=m:=\lfloor C_0 \log(n)\rfloor$, it holds that 
			$e_{G_3}(X,Y)> 0.2 m^2$.
			\item[(G4)] For every $v \in V \setminus (S^*\cup \{x^\ast\})$
			it holds that $d_{G_4}(v,S^*) \geq 4 C_0 \log(n)$.
		\end{itemize}
		
		Details on why such a partiton exists are given in the strategy discussion.
		
		\smallskip
		
		\textbf{Stage III:} Playing on $G_1$, Waiter 	
		forces Client's graph to satisfy that $|N_C(v)\cap N_C(w)\cap V_1|\geq \alpha n$ holds for every $v,w \in V$. The details on how Waiter can achieve this goal will be given in the strategy discussion. Waiter proceeds with Stage IV.
		
		\medskip
		
		\textbf{Stage IV:} Playing on $G_2$, Waiter 
		forces Client's graph to satisfy that for every $v \in V_1$ it holds that $d_C(v,V_2) \geq 40 \lfloor \gamma n \rfloor $. The details on how Waiter can achieve this goal will be given in the strategy discussion. Waiter proceeds with Stage V.
		
		\medskip
		
		\textbf{Stage V:} Playing on $G_3$, Waiter 
		forces Client's graph to satisfy that between every two disjoint sets 
		$A\subset V_1$ and $B\subset V$ of size
		$\lfloor C_0\log(n) \rfloor$ there is an edge in $C$. The details on how 
		Waiter can achieve this goal will be given in the strategy 
		discussion. 
		
		\medskip
		
		\textbf{Stage VI:} Playing on $G_4$, for every 
		$v\in V\setminus (S^\ast \cup \{x^\ast\})$
		Waiter ensures that by the end of this stage it holds that 
		$d_C(v,S^*) \geq 2 \lfloor C_0 \log(n) \rfloor$. 
		The details on how 
		Waiter can achieve this goal will be given in the strategy 
		discussion. 
		
		\smallskip
		
		Now, before coming to the strategy discussion, let us first check that,
		if Waiter can follow the strategy without forfeiting the game 
		and can reach all the described goals,
		the final Client's graph $C$ fulfils the properties (1)--(5) of 
		Theorem~\ref{thm:universal_new}. 
		
		Property~(1) is true by the initial partition of $V$.
		(2.a) and (2.b) are ensured in Stage I,
		while (2.c) is given by the outcome of  Stage VI.
		Property (3) is obtained in Stage III,
		and property (4) is done in Stage V.
		Moreover, we can distribute the collection of cliques $\mathcal{K}$ arbitrarily into $\mathcal{K}=\mathcal{K}_{good}\cup \mathcal{K}_{bad}$
		such that $|\mathcal{K}_{bad}|=\lfloor \gamma n \rfloor$. 
		This way, (5.a) holds trivially. For property (5.b) note that by the outcome of Stage IV every vertex $v\in V_1$ satisfies $d_C(v,V_2) > 40 \lfloor \gamma n \rfloor$. Moreover,
		using the matchings from Stage II.b,
		we also obtain such a bound for every vertex $v$ belonging to $\mathcal{K}_{bad}$. Finally,
		property (5.c) follows from Stage II.b.
		
		\medskip
		
		\textbf{Strategy discussion:}
		We discuss all of the five stages separately. Note that the boards of these different stages are disjoint from each other.
		
		\smallskip	
		
		\textbf{Stage I:} Waiter can easily follow this strategy for large 
		enough $n$.
		
		\smallskip
		
		\textbf{Stage II:} Waiter can force a $K_5$-factor 
		on each $K_n[V_{2,j}]$ by Theorem~1.2 in~\cite{dvovrak2023waiter}. 
		Let
		$\mathcal{K} = \{K_{i,j}:~ i\in [\lfloor \gamma n \rfloor],
		~j\in [100]\}$
		be the set of cliques in the union of these $K_5$-factors.
		Additionally, for every $j_1\neq j_2$ and 
		$i_1,i_2\in [\lfloor \gamma n \rfloor]$, 
		let $G_{i_1,j_1,i_2,j_2}\subset K_n$ be the complete 
		bipartite graph between $V(K_{i_1,j_1})$ and $V(K_{i_2,j_2})$. 
		Then all edges
		of $G_{i_1,j_1,i_2,j_2}$ are still free when Waiter enters Stage~II, and for 
		different tuples $(i_1,j_1,i_2,j_2)\neq (i_1',j_1',i_2',j_2')$ 
		the graphs $G_{i_1,j_1,i_2,j_2}$ and $G_{i_1',j_1',i_2',j_2'}$
		are edge-disjoint. Hence, Waiter can apply the strategy
		from Lemma~\ref{lem:matching.K55} to each of the graphs 
		$G_{i_1,j_1,i_2,j_2}$ separately, and thus create a matching of size 5.
		
		\textbf{Preparatory step:}
		In order to show that there is a partition 
		$G_1\cup G_2\cup G_3\cup G_4$
		of the graph induced by 
		$E_F(K_n) \setminus E_{K_n}(V_2)$  
		as desired, we show that a randomly chosen partition 
		$G_1\cup G_2\cup G_3\cup G_4$
		a.a.s.~satisfies the properties (G1)--(G4).
		To be more precise, for each edge $e \in E_F(K_n) \setminus E_{K_n}(V_2)$ 
		we decide independently with equal probability $1/4$ in which of the graphs 
		$G_1,G_2,G_3,G_4$ it will be included.
		We consider each of the desired properties separately.
		
		\begin{itemize}
			\item[(G1)] Let $v,w \in V$, then
			$|N_{K_n}(v)\cap N_{K_n}(w)\cap V_1| 
			= |V_1\setminus \{v,w\}|$. Since so far
			edges intersecting $V_1$ have only been claimed in
			Stage I, we conclude that
			$$|N_{F}(v)\cap N_{F}(w)\cap V_1| 
			\geq n - 500\gamma n - 2 - 50C_0\log(n) 
			> 0.99n ,$$
			by the choice of $\gamma$, for large $n$.
			For the random variable 
			$X^{G_1}_{v,w} = |N_{G_1}(v)\cap N_{G_1}(w)\cap V_1|$
			we have 
			$X^{G_1}_{v,w} \sim \text{Bin} \left(%
			|N_{F}(v)\cap N_{F}(w)\cap V_1|, \frac{1}{16} \right)$ 
			with expectation 
			$\mathbb E(X^{G_1}_{v,w}) >  0.06n$. 
			Applying Chernoff (Lemma~\ref{lem:Chernoff}) and
			union bound we find that (G1) fails with probability $o(1)$.	
			
			\item[(G2)] Checking (G2) can be done analogously,
			using that for every $v\in V_1$, we have
			$d_F(v,V_2)\geq |V_2| - 50C_0\log(n) > 499\gamma n$.
			
			\item[(G3)]	This can be checked analogously to property (G3)
			in the proof of Theorem~\ref{thm:MB.tree.universal},
			using that for every
			every disjoint sets $X\subset V_1,Y\subset V$ 
			of size $|X|=|Y|=m=\lfloor C_0 \log(n)\rfloor$, 
			it holds that $e_{F}(X,Y) \geq 
			|X|\cdot |Y| - 50C_0\log(n) \geq 0.9 m^2$.
			
			\item[(G4)]	This can be checked analogously to property (G2),
			using that $d_F(v,S^\ast)=|S^\ast|$ for every 
			$v\notin S^\ast\cup \{x^\ast\}$.
		\end{itemize}

		\textbf{Stage III:} In this stage Waiter only offers edges of $G_1$ 	
		which are incident with $V_1$. Waiter plays the strategy from 
		Lemma~\ref{lem:large.common.degree} with $\beta=0.05$, 
		and $N_{\{v,w\}}:=N_{G_1}(v)\cap N_{G_1}(w)\cap V_1$
		for every set $\{v,w\}$ of two vertices in $V$.
		By Lemma~\ref{lem:large.common.degree}, Waiter then ensures that
		$|N_C(v)\cap N_C(w)\cap N_{\{v,w\}}| 
		\geq \frac{\beta n}{500} \geq \alpha n$ holds 
		for every $v,w\in V$.
		
		\textbf{Stage IV:} 
		By property (G2) it holds that 
		$d_{G_2}(v,V_2) > 80 \lfloor \gamma n \rfloor$ for every 
		$v \in V_1$.  With a simple pairing strategy,
		Waiter reaches the described goal.
		
		\textbf{Stage V:} Waiter reaches the described goal
		by an application of Theorem~\ref{thm:WC_transversal}.
		To be more precise, let
		$$
		\mathcal{F} := \{E_{G_3}(A,B): A \subset V_1, B \subset V,
		|A|=|B| = \lfloor C_0 \log(n) \rfloor, A \text{ and } B \text{ disjoint} \}.
		$$
		Using (G3) and the choice of $C_0$,
		we obtain
		$\sum_{F \in \mathcal{F}} 2^{-|F|+1} 
		\leq n^{2m} \cdot 2^{-0.2m^2+1} = o(1)$
		analogously to the discussion of Maker's strategy. In particular, 
		Waiter can force Client to claim an
		element of each edge set in $\mathcal{F}$.
		
		\textbf{Stage VI:} 
		By property (G4) it holds that 
		$d_{G_4}(v,S^\ast) > 4C_0\log(n)$ for every 
		$v \in V\setminus (S^\ast \cup \{x^\ast\})$.  
		Again, with a simple pairing strategy,
		Waiter reaches the described goal.
	\end{proof}

	\section{Concluding remarks}\label{sec:concluding}
	
	Though we presented no $\Delta$ such that Maker in the $(1:1)$ game on $K_n$ cannot obtain a graph which is universal for spanning trees of degree $\Delta$, we believe that the degree order $n/\log n$ in   Theorem~\ref{thm:MB.tree.universal} is optimal. More precisely, we pose the following conjecture.
	
	\begin{conjecture}
		There exists a constant $C>0$ such that the following holds for every large enough integer $n$.
		In the $(1:1)$ Maker-Breaker game on $K_n$,
		Breaker has a strategy such that Maker cannot build a graph which contains 
		a copy of every tree $T$ with $n$ vertices
		and maximum degree $\Delta(T)\leq \frac{Cn}{\log(n)}$.
	\end{conjecture}
	
	In contrast, we believe that the maximum degree order in Theorem~\ref{thm:WC.tree.universal} can be improved. 
	
	\begin{conjecture}
		There exists a constant $c >0$ such that the following holds for every large enough integer $n$.
		In the $(1:1)$ Waiter-Client game on $K_n$,
		Waiter has a strategy to force Client to claim a graph which contains 
		a copy of every tree $T$ with $n$ vertices
		and maximum degree $\Delta(T)\leq c n$.
	\end{conjecture}
	
	\subsection{Tree universality in Client-Waiter games}
	
	Together with Picker-Chooser games Beck introduced also Chooser-Picker games (cf.~\cite{beck2008combinatorial}), later studied under the name Client-Waiter (cf.~\cite{deanCW}, \cite{hefetzCW}). In a $(1:1)$ Client-Waiter game on some hypergraph $\mathcal{H} = (\mathcal{X},\mathcal{F})$ Waiter picks $2$ elements of the board $\mathcal{X}$ and offers them to Client. Client chooses one of them for himself and returns the rest to Waiter. If there is only one element left in the last round, it goes to Client. Client wins if she fully claims a winning set $F \in \mathcal{F}$; otherwise, Waiter wins. 
	
	It is well known (and observed already by Beck) that the Erd\H os-Selfridge criterion can be adapted for Client-Waiter games 
	%
	and if 
	$$\sum_{F\in \mathcal{F}} 2^{-|F|} < 1,$$
	then in the $(1:1)$ Client-Waiter game on $(X,\mathcal{F})$, Client has a strategy to claim at least one element in each of the sets in $\mathcal{F}$. 
	As in Maker-Breaker games, we can use this criterion to prove that Client can build a good expander in $K_n$. Roughly speaking, we define a family $\mathcal{F}$ of edge sets in $K_n$ with the property that if Client has at least one edge in every set from $\mathcal{F}$, then every vertex set of her graph has a big neighborhood, and there is a Client's edge between every pair of not too small sets. 
	In view of expander properties from~\cite{johannsen2013expanders}, one can deduce that Client in the $(1:1)$ game on $K_n$ can build a graph that contains copies of all trees $T$ with $\Delta(T) \leq \frac{c n^{1/3}}{\log(n)}$. We can further relax the last inequality to $\Delta(T) \leq \frac{c n^{1/2}}{\log(n)}$, if we apply a result from \cite{han2023expanders}. Unfortunately, we do not see how to adapt our proof of Theorem~\ref{thm:WC.tree.universal} to the Client-Waiter version; still, we suspect that the following is true.
	
	\begin{conjecture}
		There exists a constant $c>0$ such that the following holds for every large enough integer $n$.
		In the $(1:1)$ Client-Waiter game on $K_n$,
		Client has a strategy to build a graph which contains 
		a copy of every tree $T$ with $n$ vertices
		and maximum degree $\Delta(T)\leq \frac{cn}{\log(n)}$.
	\end{conjecture}
	
	The degree order in the above conjecture cannot be improved since it is known \cite{fixedtree} that there exists a constant $C>0$ and a tree $T$ with $n$ vertices
	and maximum degree $Cn/\log(n)$ such that Client cannot build a copy of $T$ in $K_n$.
	
	\subsection{Tree universality in Avoider-Enforcer games}
	
	Finally, let us mention another class of positional games, called Avoider-Enforcer or Avoider-Forcer games (cf.~\cite{beck2008combinatorial}, \cite{hefetz2014positional}). For simplicity, let us focus on the symmetric and so-called strict version only. In a $(1:1)$ Avoider-Enforcer game on some hypergraph $\mathcal{H} = (\mathcal{X},\mathcal{F})$ Avoider (who starts) and Enforcer select in turns one (not yet selected) element of the board $\mathcal{X}$, until all elements are selected. Enforcer wins if at the end of the game all elements of at least one set $F \in \mathcal{F}$ belong to Avoider; otherwise Avoider is the winner. Lu \cite{Lu1991} proved that the Erd\H os-Selfridge condition on $\mathcal{H}$ implies that Avoider has a winning strategy in the $(1:1)$ Avoider-Enforcer game on $\mathcal{H}$. Furthermore, it is known that the assertion holds also when Enforcer starts the game. In view of that, we can say that the Erd\H os-Selfridge condition implies that the second player in the $(1:1)$ Avoider-Enforcer game on $\mathcal{H}$ can force the first player to claim at least one element in each of the sets in $\mathcal{H}$. 
	It is now enough to add expander properties from~\cite{han2023expanders} to infer that Enforcer in the $(1:1)$ game on $K_n$ can force Avoider to build a graph that contains copies of all trees $T$ with  $\Delta(T) \leq \frac{c n^{1/2}}{\log(n)}$. It seems challenging to improve this result.  
	\subsection{Waiter-Client minimum pair degree game}
	In Lemma~\ref{lem:large.common.degree}, we prove that Waiter can force Client in the $(1:1)$ Waiter-Client game on $K_n$ to claim a graph where each pair of vertices has pair degree $\alpha n$ for some suitable $\alpha$. We believe that this result might be of independent interest. Furthermore, we want to pose the following problem:
	\begin{problem}
		Find the maximum $\alpha$ such that for every large enough $n$ Waiter has a strategy in a $(1:1)$ Waiter-Client game on $K_n$ to force Client to claim a graph with the following property: for any two vertices $v,w \in V(K_n)$ we have $|N_C(v) \cap N_C(w)| \geq \alpha n$.
	\end{problem}
	Note that in the Maker-Breaker version, Breaker can easily ensure that $|N_M(v) \cap N_M(w)| = 0$ for two fixed vertices $v, w \in V(K_n)$.
	\bibliographystyle{amsplain}
	\bibliography{references}

\end{document}